\documentclass[notitlepage, 11pt, letterpaper]{amsart}

\usepackage{amsmath}
\usepackage{amssymb}
\usepackage{graphicx}
\usepackage{enumerate}
\usepackage{bbm}
\usepackage[table]{xcolor}
\usepackage{latexsym}
\usepackage{mathtools}
\usepackage{braket}
\usepackage{diagbox}
\usepackage[aligntableaux=top,boxsize=0.2em]{ytableau} 
\usepackage[foot]{amsaddr}
\usepackage[
    backend=biber,
    style=numeric-comp,
    sorting=none,
    doi=false,
    url=false,
    isbn=false
]{biblatex}
\usepackage[
    unicode=true,
    colorlinks=true,
    urlcolor=blue!75!black,
    linkcolor=blue!75!black,
    citecolor=blue!75!black,
    pdfusetitle,
    bookmarks=true,
    bookmarksnumbered=true,
    bookmarksopen=true,
    breaklinks=true
]{hyperref}
\DeclareFieldFormat{eprint}{#1}
\DeclareFieldFormat{eprint:arxiv}{arXiv:\,#1}

\usepackage{thmtools}
\usepackage{thm-restate}
\usepackage{amsthm}
\usepackage[capitalise]{cleveref} 
\theoremstyle{plain}
\newtheorem{thm}{Theorem}
\newtheorem{lem}[thm]{Lemma}
\newtheorem{cor}[thm]{Corollary}

\newtheorem{prop}[thm]{Proposition}
\theoremstyle{definition}
\newtheorem{rem}[thm]{Remark}

\makeatletter
\if@cref@capitalise
\Crefname{thm}{Theorem}{Theorems}
\Crefname{lem}{Lemma}{Lemmas}
\Crefname{cor}{Corollary}{Corollaries}
\Crefname{conj}{Conjecture}{Conjectures}
\Crefname{defn}{Definition}{Definitions}
\Crefname{prop}{Proposition}{Propositions}
\Crefname{exam}{Example}{Examples}
\Crefname{prob}{Problem}{Problems}
\Crefname{rem}{Remark}{Remarks}
\else
\crefname{thm}{theorem}{theorems}
\crefname{lem}{lemma}{lemma}
\crefname{cor}{corollary}{corollaries}
\crefname{conj}{conjecture}{conjectures}
\crefname{defn}{definition}{definitions}
\crefname{prop}{proposition}{propositions}
\crefname{exam}{example}{examples}
\crefname{prob}{problem}{problems}
\crefname{rem}{remark}{remarks}
\fi
\makeatother
\usepackage{crossreftools}
\let\ORGhypersetup\hypersetup
\protected\def\hypersetup{\ORGhypersetup}
\DeclareMathOperator{\diag}{diag}
\DeclareMathOperator{\Tr}{Tr}
\DeclareMathOperator{\Horn}{Horn}
\DeclareMathOperator{\Tetra}{Tetra}
\DeclareMathOperator{\Lie}{Lie}
\DeclareMathOperator{\Mat}{Mat}
\DeclareMathOperator{\End}{End}
\DeclareMathOperator{\Hom}{Hom}
\DeclareMathOperator{\BC}{BC} 
\DeclareMathOperator{\rank}{rank}
\DeclareMathOperator{\eig}{eig}
\DeclareMathOperator{\id}{id}

\DeclareMathOperator*{\expect}{\mathbb{E}}
\newcommand{\card}[1]{\left|#1\right|}
\newcommand{\diff}{\mathrm{d}}
\newcommand{\orbit}{\mathcal O}
\newcommand{\polyk}[1]{(k\!+\!1)^{#1}}
\newcommand{\norm}[1]{\left\lVert#1\right\rVert}
\newcommand{\abs}[1]{\left|#1\right|}
\newcommand{\kl}[2]{D\!\left(#1\middle\|#2\right)} 

\newcommand{\SW}[1]{P_{#1}}
\newcommand{\herm}[1]{\mathcal H_{#1}}
\newcommand{\hermp}[1]{\mathcal H^{+}_{#1}}

\newcommand{\undomweights}[1]{\mathbb Z^{\geq}_{#1}}
\newcommand{\undomweightsp}[1]{\mathbb Z^{+}_{#1}}
\newcommand{\eigs}[1]{\mathbb R^{\geq}_{#1}}
\newcommand{\eigsp}[1]{\mathbb R^{+}_{#1}}
\newcommand{\ot}{\otimes} 
\newcommand{\vp}{\varphi} 
\newcommand{\projl}{\Pi_{(\ya \yb) \yd}^{\yc, \ye}}
\newcommand{\projr}{\Pi_{\ya (\yb \yd)}^{\ye, \yf}}

\newcommand{\ea}{a}
\newcommand{\eb}{b}
\newcommand{\ec}{c}
\newcommand{\ed}{d}
\newcommand{\ee}{e}
\newcommand{\ef}{f}
\newcommand{\ex}{x}
\newcommand{\ey}{y}
\newcommand{\ez}{z}

\newcommand{\tea}{\Tr[\ea]}
\newcommand{\teb}{\Tr[\eb]}
\newcommand{\tec}{\Tr[\ec]}
\newcommand{\ted}{\Tr[\ed]}
\newcommand{\tee}{\Tr[\ee]}
\newcommand{\tef}{\Tr[\ef]}
\newcommand{\tex}{\Tr[\ex]}
\newcommand{\tey}{\Tr[\ey]}

\newcommand{\ya}{\alpha}
\newcommand{\yb}{\beta}
\newcommand{\yc}{\gamma}
\newcommand{\yd}{\delta}
\newcommand{\ye}{\epsilon}
\newcommand{\yf}{\phi}
\newcommand{\yx}{\lambda}

\newcommand{\sya}{|\ya|}
\newcommand{\syb}{|\yb|}
\newcommand{\syc}{|\yc|}
\newcommand{\syd}{|\yd|}
\newcommand{\sye}{|\ye|}
\newcommand{\syf}{|\yf|}
\newcommand{\syx}{|\yx|}

\newcommand{\tightsextuple}[1]{
\begingroup
\arraycolsep=1.5pt
\renewcommand{\arraystretch}{0.3}
\begin{array}{@{}cccccc@{}}#1\end{array}
\endgroup
}

\newcommand{\tightcurlysixj}[1]{
\begingroup
\arraycolsep=1.5pt
\renewcommand{\arraystretch}{0.3}
\left\{\!\begin{array}{ccc}
    #1
\end{array}\!\right\}
\endgroup
}
\newcommand{\tightparensixj}[1]{
\begingroup
\arraycolsep=1.5pt
\renewcommand{\arraystretch}{0.3}
\left(\!\begin{array}{ccc}
    #1
\end{array}\!\right)
\endgroup
}

\newcommand{\abcdefsymbol}{\tightcurlysixj{\alpha & \beta & \gamma \\ \delta & \epsilon & \phi}}
\newcommand{\abcdefsymbolk}{\tightcurlysixj{\alpha_k & \beta_k & \gamma_k \\ \delta_k & \epsilon_k & \phi_k}}
\newcommand{\abcdefsymbolcefstarred}{\tightcurlysixj{\alpha & \beta & \gamma^* \\ \delta & \epsilon^* & \phi^*}}

\newcommand{\tightsinglesubstack}[2]{
\begingroup
    \renewcommand{\arraystretch}{0.3}
    \begin{array}{c@{}} #1 \\ {\scriptscriptstyle #2} \end{array}
\endgroup
}
\newcommand{\tightdoublesubstack}[3]{
\begingroup
    \renewcommand{\arraystretch}{0.3}
    \begin{array}{c@{}} #1 \\ {\scriptscriptstyle #2} \\ {\scriptscriptstyle #3} \end{array}
\endgroup
}
\newcommand{\tighttriplesubstack}[4]{
\begingroup
    \renewcommand{\arraystretch}{0.3}
    \begin{array}{c@{}} #1 \\ {\scriptscriptstyle #2} \\ {\scriptscriptstyle #3} \\ {\scriptscriptstyle #4} \end{array}
\endgroup
}

\begin{document}

\title[Tetrahedral Horn \& U(N) 6j asymptotics]%
{The tetrahedral Horn problem and \\ asymptotics of \texorpdfstring{$\boldsymbol{U(n)}$ $\boldsymbol{6j}$}{U(n) 6j} symbols}

\author{Anton Alekseev\texorpdfstring{$^1$}{}}
\email{anton.alekseev@unige.ch}
\author{Matthias Christandl\texorpdfstring{$^2$}{}}
\email{christandl@math.ku.dk}
\author{Thomas C. Fraser\texorpdfstring{$^2$}{}}
\email{tcf@math.ku.dk}

\address{\texorpdfstring{$^1$}{}Department of Mathematics, University of Geneva, 7-9 rue du Conseil Général, 1211 Geneva 4, Switzerland}
\address{\texorpdfstring{$^2$}{}Department of Mathematical Sciences, University of Copenhagen, Universitetsparken 5, 2100 Denmark}

\date{\today}

\begin{abstract}
    Horn's problem is concerned with characterizing the eigenvalues $(\ea,\eb,\ec)$ of Hermitian matrices $(A,B,C)$ satisfying the constraint $A+B=C$ and forming the edges of a triangle in the space of Hermitian matrices.
    It has deep connections to tensor product invariants, Littlewood-Richardson coefficients, geometric invariant theory and the intersection theory of Schubert varieties.
    This paper concerns the \textit{tetrahedral Horn problem} which aims to characterize the tuples of eigenvalues $(\ea,\eb,\ec,\ed,\ee,\ef)$ of Hermitian matrices $(A,B,C,D,E,F)$ forming the edges of a tetrahedron, and thus satisfying the constraints $A+B=C$, $B+D=F$, $D+C=E$ and $A+F=E$.

    Here we derive new inequalities satisfied by the Schur-polynomials of such eigenvalues and, using eigenvalue estimation techniques from quantum information theory, prove their satisfaction up to degree $k$ implies the existence of approximate solutions with error $O(\ln k / k)$.
    Moreover, the existence of these tetrahedra is related to the semiclassical asymptotics of the $6j$-symbols for the unitary group $U(n)$, which are maps between multiplicity spaces that encode the associativity relation for tensor products of irreducible representations.
    Using our techniques, we prove the asymptotics of norms of these $6j$-symbols are either inverse-polynomial or exponential depending on whether there exists such tetrahedra of Hermitian matrices.
\end{abstract}

\maketitle
\tableofcontents

\section{Introduction}

\subsubsection*{The Horn problem}
Let $\herm{n}$ denote the set of $n\times n$ Hermitian matrices, $\eigs{n}$ the corresponding set of ordered eigenvalues and $\eig : \herm{n} \to \eigs{n}$ the assignment map.
The Horn problem is a classical linear algebra problem asking to determine the set of possible eigenvalues for $A,B,C \in \herm{n}$ satisfying the constraint $A+B=C$, denoted by
\begin{equation}
    \Horn(n) \coloneqq \left\{ (\eig A, \eig B, \eig C) \middle| A + B = C \right\}.
\end{equation}
Although the problem has been investigated by numerous authors for over a century \cite{weyl1912asymptotische,fan1949theorem,lidskii1953proper,wielandt1955extremum,thompson1971eigenvalues,thompson1970eigenvalues}, it is named after \citeauthor{horn1962eigenvalues} who conjectured a solution to the problem in the form of an inductively defined, finite set of homogeneous linear inequalities \cite{horn1962eigenvalues} together with the trace equality
\begin{equation}
    \Tr[c] = \Tr[a] + \Tr[b],
\end{equation}
where $\tex \coloneqq \sum_{i=1}^{n} \ex_i$ for $\ex \in \eigs{n}$.
In \citeyear{klyachko1998stable}, \citeauthor{klyachko1998stable} provided a solution to the Horn problem \cite{klyachko1998stable} which, when combined with a proof of the so-called saturation conjecture, due to \citeauthor{knutson1999honeycomb} \cite{knutson1999honeycomb}, constituted a proof of Horn's conjectured solution \cite{fulton1998eigenvalues} (see also \cite{alekseev2017symplectic}).
These linear inequalities, now known as the \textit{Horn inequalities} describe the convex polyhedral cone of admissible values for $(\ea,\eb,\ec) \in \Horn(n)$ \cite{bhatia2001linear,knutson2000symplectic,fulton2000eigenvalues,niculescu2006convex}.

\subsubsection*{The tetrahedral Horn problem}
In this paper, we consider the problem of determining the set of possible eigenvalues for $n\times n$ Hermitian matrices $A,B,C,D,E,F \in \herm{n}$ subject to four constraints, denoted by
\begin{equation}
    \label{eq:tet_cond}
    \Tetra(n) \coloneqq \left\{
        \bigg(
        \begin{array}{c}
            \eig A, \eig B, \eig C, \\
            \eig D, \eig E, \eig F
        \end{array}
        \bigg)
    \middle|
        \begin{array}{cc}
            A+B=C, & B+D=F, \\
            D+C=E, & A+F=E
        \end{array}
    \right\}.
\end{equation}
We refer to this eigenvalue problem as the \textit{tetrahedral} Horn problem because these constraints are naturally associated to the four faces of a tetrahedron in the vector space $\herm{n}$, as depicted in \cref{fig:tetrahedron_labelling_scheme}.
\begin{figure}[ht]
    \centering
    \includegraphics[]{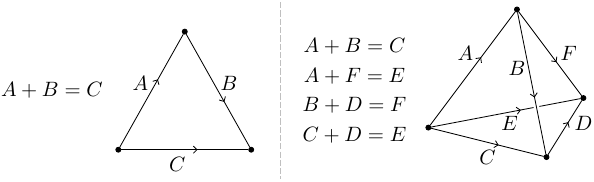}
    \caption{(Left) The Horn problem concerns the eigenvalues of three Hermitian matrices, $A,B,$ and $C$, forming a triangle in $\herm{n}$.
    (Right) The tetrahedral Horn problem concerns the eigenvalues of six Hermitian matrices, $A,B,C,D,E,$ and $F$, forming a tetrahedron in $\herm{n}$.}
    \label{fig:tetrahedron_labelling_scheme}
\end{figure}

Beyond the trace conditions,
\begin{align}
    \label{eq:tet_trace_cond}
    \begin{split}
        \Tr[c] &= \Tr[a] + \Tr[b], \quad
        \Tr[f] = \Tr[b] + \Tr[d], \\
        \Tr[e] &= \Tr[a] + \Tr[f], \quad
        \Tr[e] = \Tr[c] + \Tr[d],
    \end{split}
\end{align}
there are additional inequalities necessarily satisfied by $(\ea,\eb,\ec,\ed,\ee,\ef) \in \Tetra(n)$.
First, observe that the Horn inequalities \cite{horn1962eigenvalues} apply to the eigenvalues of the triples of Hermitian matrices associated to each face: i.e., to $(\ea,\eb,\ec)$, $(\eb,\ed,\ef)$, $(\ea,\ef,\ee)$, and $(\ec,\ed,\ee)$.
Second, there exist inequalities known as submodular or strongly subadditive inequalities which constrain the eigenvalues of the quadruple of Hermitian matrices $(Y,X+Y,Y+Z,X+Y+Z)$ whenever $X,Y,Z$ are taken to be non-negative Hermitian matrices \cite{berndt2015hlawka, tie2011rearrangement, sagnol2013approximation, bouhtou2008optimization, friedland2013submodular, lewin2014family, sra2017determinantoverflow}.
These inequalities are readily applicable to the tetrahedral Horn problem under the identification ($X=A$, $Y=B$ and $Z=D$) wherein they constrain, respectively, the eigenvalues $(\eb,\ec,\ee,\ef)$ of the matrices $(B,C,E,F)$ while ignoring the eigenvalues of $A$ and $D$.
More scarce are inequalities which additionally incorporate the eigenvalues $(\ea,\ed)$ of the matrices $(A,D)$ such as Minkowski-like determinant inequality \cite[Section 6.2]{niculescu2023functions} (originally due to \citeauthor{courtade2016strengthening} \cite{courtade2016strengthening, sra2016reverseoverflow}), or applications of Ptolemny's inequality \cite{schoenberg1952remark} or the Cayley-Menger determinant inequality \cite[Equation 5.2]{sitharam2018handbook} to the space of $n\times n$ Hermitian matrices equipped with the Hilbert-Schmidt inner product.
While each of these previously known inequalities can be viewed as necessary conditions for the tetrahedral Horn problem, they remain insufficient for the tetrahedral Horn problem when $n\geq 3$.

\subsubsection*{The case of \texorpdfstring{$2 \times 2$}{2 by 2} matrices}
Beyond the superficial connection between the imposed constraints and the resulting geometry depicted in \cref{fig:tetrahedron_labelling_scheme}, in the special case of $2\times 2$ Hermitian matrices, this connection is much deeper.
Without loss of generality, it is always possible to restrict attention to the case of \textit{traceless} Hermitian matrices (see \cref{sec:symmetries} for a proof).
Every $2\times 2$ traceless Hermitian matrix can be identified by a three-dimensional vector $v = (v_1, v_2, v_3) \in \mathbb R^{3}$ using the formula:
\begin{equation}
    X = \begin{pmatrix} v_3 & v_1-iv_2 \\ v_1+iv_2 & -v_3 \end{pmatrix},
\end{equation}
in which case, the eigenvalues of $X$ are given by
\begin{equation}
    \eig(X) = (\ell_x, -\ell_x),
\end{equation}
where $\ell_x \coloneqq \sqrt{v_1^2 + v_2^2 + v_3^2}$ is the Euclidean length of the vector $v$.
As a consequence, the solution to the original Horn problem for $n = 2$ is equivalent to the space of possible side-lengths of \textit{triangles} in $\mathbb R^{3}$ \cite{hausmann1996polygon,freidel2010holomorphic,knutson2000symplectic}.
This geometric correspondence generalizes to the tetrahedral Horn problem where the solution for $n=2$ is equivalent to the space of possible side-lengths of \textit{tetrahedra} in $\mathbb R^{3}$.
This geometric knowledge about the solution space then leads to an algebraic solution to the tetrahedral Horn problem for $n = 2$ \cite{richardson1902trigonometry,wirth2009edge};
specifically, the space of possible side-lengths of tetrahedra in $\mathbb R^3$ is described by the conjunction of the triangle inequalities for each of the four faces,
\begin{align}
    \label{eq:quadruplet_triangle_constraints}
    \begin{split}
        \abs{\ell_a-\ell_b} &\leq \ell_c \leq \ell_a+\ell_b, \qquad
        \abs{\ell_b-\ell_d} \leq \ell_f \leq \ell_b+\ell_d, \\
        \abs{\ell_d-\ell_c} &\leq \ell_e \leq \ell_d+\ell_c, \qquad
        \abs{\ell_a-\ell_f} \leq \ell_e \leq \ell_a+\ell_f,
    \end{split}
\end{align}
together with a non-linear constraint of the form \cite{sitharam2018handbook}:
\begin{equation}
    \label{eq:volume_positive_constraint}
    \mathrm{CM}(\ell_a,\ell_b,\ell_c,\ell_d,\ell_e,\ell_f) = \frac{1}{288}\det
    \begin{pmatrix}
        0 & \ell_a^2 & \ell_c^2 & \ell_e^2 & 1 \\
        \ell_a^2 & 0 & \ell_b^2 & \ell_f^2 & 1 \\
        \ell_c^2 & \ell_b^2 & 0 & \ell_d^2 & 1 \\
        \ell_e^2 & \ell_f^2 & \ell_d^2 & 0 & 1 \\
        1 & 1 & 1 & 1       & 0
    \end{pmatrix}
    \geq 0.
\end{equation}
The quantity $\mathrm{CM}(\ell_a,\ell_b,\ell_c,\ell_d,\ell_e,\ell_f)$ is known as a Cayley-Menger determinant; if $\mathrm{CM}(\ell_a,\ell_b,\ell_c,\ell_d,\ell_e,\ell_f)$ is positive and if the triangle inequalities above are satisfied, the value of $\mathrm{CM}(\ell_a,\ell_b,\ell_c,\ell_d,\ell_e,\ell_f)$ is equal to squared volume of the corresponding tetrahedron \cite[Equation 5.2]{sitharam2018handbook}.

Before proceeding, we emphasize the strong qualitative difference between the tetrahedral Horn problem and the original Horn problem; while $\Horn(n)$ is a convex polyhedral cone, $\Tetra(n)$ is neither convex nor polyhedral.
\Cref{fig:tetrahedron_moduli_space_example} visualizes this distinction by depicting a three-dimensional slice of $\Tetra(2)$.
\begin{figure}
    \begin{center}
        \includegraphics[width=0.7\textwidth]{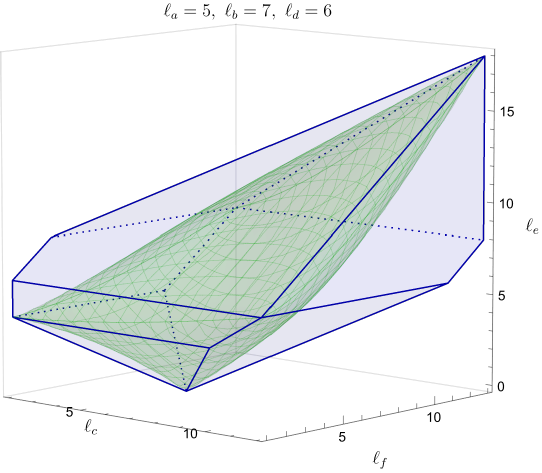}
    \end{center}
    \caption{
        The interior green region depicts a slice of the space of side-lengths $(\ell_a,\ell_b,\ell_c,\ell_d,\ell_e,\ell_f)$ of tetrahedra where $(\ell_a,\ell_b,\ell_d) = (5,7,6)$ and thus depicts a slice of $\Tetra(2)$.
        The exterior blue polytope depicts the values of $(\ell_c,\ell_e,\ell_f)$ which merely satisfy the triangle inequalities, while the interior green region additionally satisfies the Cayley-Menger determinant inequality.
    }
    \label{fig:tetrahedron_moduli_space_example}
\end{figure}

\subsubsection*{Our contributions}
In this paper, we provide new necessary inequalities for the tetrahedral Horn problem and moreover prove their sufficiency.
In order to describe these conditions, we make the following simplifying assumption.
Denote by $\hermp{n} \subset \herm{n}$ the set of non-negative Hermitian matrices, and by $\eigsp{n} \subset \eigs{n}$ the corresponding set of possible eigenvalues.
We claim that the tetrahedral Horn problem can be reduced to the problem of characterizing the subset of non-negative eigenvalues:
\begin{equation}
    \Tetra^+(n) \coloneqq (\eigsp{n})^6 \cap \Tetra(n).
\end{equation}
This restriction to non-negative eigenvalues is justified because $\Tetra(n)$ is closed under certain affine transformations (for more details see \cref{cor:non_negativity} in \cref{sec:symmetries}).
Additionally, we define the following distance measure:
\begin{equation}
    D\!\tightparensixj{\ea & \eb & \ec \\ \ed & \ee & \ef} \coloneqq
    \min_{\substack{C=A+B \\ F=B+D \\ E = C+D \\ E = A+F}}
    \begin{array}{c}
        [
        \norm{ \ea - \eig A }_{1}^2 + \norm{ \eb - \eig B }_{1}^2 + \norm{ \ec - \eig C }_{1}^2 + \\
        \norm{ \ed - \eig D }_{1}^2 + \norm{ \ee - \eig E }_{1}^2 + \norm{ \ef - \eig F }_{1}^2
        ]^{\frac{1}{2}},
    \end{array}
\end{equation}
where $\norm{x}_{1} \coloneqq \sum_{i=1}^{n} \abs{x_i}$ is a distance measure which emerges naturally from the probabilistic techniques used in this paper.
Note that $D\!\tightparensixj{\ea & \eb & \ec \\ \ed & \ee & \ef} \geq 0$ with $D\!\tightparensixj{\ea & \eb & \ec \\ \ed & \ee & \ef} = 0$ only if $(\ea, \eb, \ec, \ed, \ee, \ef) \in \Tetra(n)$.

Our methods rely on the representation theory of the group $G = U(n)$.
We denote by $\undomweights{n} = \{ (\yx_1 \geq \yx_2 \geq \cdots \geq \yx_n) \mid \yx_i \in \mathbb{Z} \}$ the set of dominant weights for the group $G=U(n)$, and
by $\undomweightsp{n} = \{ \yx \in \undomweights{n} \mid \yx_n \geq 0\}$ the set of positive dominant weights, which can be regarded as Young diagrams.
For $\yx \in \undomweightsp{n}$, we denote the size of $\yx$ by $\syx = \sum_{i=1}^n \yx_i$, and the product of the hook lengths of boxes in $\yx$ by $H_{\yx} = \prod_{b\in \yx} h_{\yx}(b)$ (the hook length of box $b$, $h_{\yx}(b)$, is defined as the number of boxes to the right and below $b$ plus one).
The irreducible representations $V_{\yx}$ corresponding to positive dominant weights are polynomial, and they extend to the matrix algebra $\Mat_n(\mathbb{C})$.
It will be convenient to use rescaled characters of representations,
\begin{equation}
    \vp_\yx(X) \coloneq \frac{1}{H_\yx} \chi_\yx(X),
\end{equation}
extended to $X \in \Mat_n(\mathbb{C})$.
Note that if $\ex = \eig X$ are the eigenvalues of $X$, then the character $\chi_\yx(X)$ is equal to the Schur-polynomial of the eigenvalues $s_{\yx}(\ex)$ and therefore, we also define $\vp_{\yx}(\ex) \coloneqq s_{\yx}(\ex)/H_\yx$ by a slight abuse of notation.
Irreducible representations $V_\yx$ of $U(n)$ are equipped with invariant Hermitian structures, and the decomposition of their tensor products,
\begin{equation}
    V_\ya \otimes V_\yb \cong \bigoplus_\yc C_{\ya \yb}^\yc \otimes V_\yc,
\end{equation}
induces Hermitian structures of multiplicity spaces $C_{\ya \yb}^\yc \coloneqq \Hom(V_{\yc}, V_{\ya}\ot V_{\yb})$.
Note that a necessary condition for $\dim C_{\ya \yb}^\yc > 0$ is the size condition $\sya + \syb = \syc$.
The associativity isomorphism between triple tensor products $(V_\ya \ot V_\yb) \ot V_\yd \cong V_\ya \ot (V_\yb \ot V_\yd)$ induces a linear map between multiplicity spaces known as the $6j$-symbol.
Specifically, a $6j$-symbol depends on six irreducible representations, $V_{\ya}, V_{\yb}, V_{\yc}, V_{\yd}, V_{\ye}$, and $V_{\yf}$, and is the linear map between multiplicity spaces of the type
\begin{equation*}
    \abcdefsymbol :
    C^{\yc}_{\ya\yb} \ot C^{\ye}_{\yc\yd} \to C^{\ye}_{\ya\yf} \ot C^{\yf}_{\yb\yd},
\end{equation*}
corresponding to the block-wise components of the canonical isomorphism from $\bigoplus_{\yc} C^{\yc}_{\ya\yb} \ot C^{\ye}_{\yc\yd}$ to $\bigoplus_{\yf} C^{\ye}_{\ya\yf} \ot C^{\yf}_{\yb\yd}$.
Here we note a necessary condition for $\abcdefsymbol \neq 0$ is the satisfaction of the following size conditions:
\begin{align}
    \label{eq:size_conditions_tet}
    \begin{split}
        \sya + \syb &= \syc, \quad
        \syb + \syd = \syf, \\
        \sya + \syf &= \sye, \quad
        \syc + \syd = \sye.
    \end{split}
\end{align}
Before continuing, we emphasize that the representation theories of $G=U(n)$ and $G=SU(n)$, and thus the properties of their respective $6j$-symbols, are essentially the same up to a shifting of dominant weights, an reinterpretation of size conditions, and overall complex phases \cite{hall2008representations}.
For the sake of concreteness, we work exclusively with the group $G=U(n)$.

Our first result is a necessary and sufficient condition for membership in $\Tetra^{+}(n)$.
\begin{restatable}{thm}{tetiffthm}
    \label{thm:tet_iff}
    Let $(\ea,\eb,\ec,\ed,\ee,\ef) \in (\eigsp{n})^6$ satisfy the trace conditions.
    Then $(\ea,\eb,\ec,\ed,\ee,\ef) \in \Tetra^{+}(n)$
    if and only if for all triples $(\ya,\yb,\yd) \in (\undomweightsp{n})^3$,
    \begin{equation}
        \label{eq:tet_iff}
        \vp_{\ya}(\ea)^{\frac{2}{3}}\vp_{\yb}(\eb)\vp_{\yd}(\ed)^{\frac{2}{3}}
        \leq \sum_{\yc,\ye,\yf} \norm{\abcdefsymbol}_{\infty}
        \left(\vp_{\yf}(\ef)\vp_{\yc}(\ec)\vp_{\ye}(\ee)\right)^{\frac{1}{3}}.
    \end{equation}
    Furthermore, if the above inequality holds for all triples $(\ya,\yb,\yd) \in (\undomweightsp{n})^3$ with $\sya + \syb + \syd = k$, then
    \begin{equation}
        \label{eq:tet_dist_bound}
        \frac{1}{\tee} D\!\tightparensixj{\ea & \eb & \ec \\ \ed & \ee & \ef} \leq 6 \sqrt{3} n \sqrt{\frac{\ln \polyk{}}{k}}.
    \end{equation}
\end{restatable}
Note that in \cref{eq:tet_iff} the sum over $\yc, \ye, \yf$ on the right-hand side in \cref{eq:tet_iff} is finite since there are only a finite number of triples $(\yc, \ye, \yf)$ for which the size conditions in \cref{eq:size_conditions_tet} hold and thus for which the $6j$-symbol is non-vanishing.
\Cref{thm:tet_iff} establishes an infinite set of inequalities which characterize the set $\Tetra^+(n)$.
To the best of our knowledge, this the first system of inequalities which provides not only necessary but also sufficient conditions for $(\ea,\eb,\ec,\ed,\ee,\ef) \in \Tetra^+(n)$.
Furthermore, for fixed $k = \sya + \syb + \syd$, the corresponding finite set of inequalities describes an outer approximation to $\Tetra^+(n)$ which becomes exact as $k \to \infty$.

Our second result concerns the behavior of $6j$-symbols for $G=U(n)$ along convergent sequences of weights.
Our interest in the $6j$-symbols for $G=U(n)$ or $G=SU(n)$ ultimately stems from the importance of the $6j$-symbols for $G=SU(2)$, also referred to as Wigner $6j$-symbols/coefficients \cite{wigner1993matrices} or the Racah $W$-coefficients \cite{racah1942theory},
where they play a foundational role in the quantum theory of angular momentum \cite{varshalovich1988quantum,edmonds1996angular,aquilanti2012semiclassical}, the theory of spin networks \cite{penrose1971applications, baez2010physics, baez1996spin}, invariants of knots and manifolds \cite{reshetikhin1987quantized, turaev1992state}, topological quantum field theory \cite{durhuus1993topological}, elementary gates in spin network models for quantum computing \cite{marzuoli2005computing}, quantum change-point discrimination \cite{akimoto2011discrimination}, and models for quantum gravity \cite{baez1999introduction,barrett1998relativistic,barrett2003geometrical,rovelli2004quantum,moussouris1984quantum}.
\begin{restatable}{thm}{sixjestimates}
    \label{thm:6j_estimates}
    Let $(\ea, \eb, \ec, \ed, \ee, \ef) \in \Tetra^+(n)$.
    Then there exists a sequence
    $(\ya_k, \yb_k, \yc_k, \yd_k, \ye_k, \yf_k) \in (\undomweightsp{n})^6$ with $\abs{\ya_k}+\abs{\yb_k}+\abs{\yd_k} = k$ such that
    \begin{equation}
        \label{eq:young_diagram_convergence}
        \frac{1}{k} (\ya_k, \yb_k, \yc_k, \yd_k, \ye_k, \yf_k)
        \stackrel{k \to \infty}{\longrightarrow}
        \frac{1}{\Tr[e]} (\ea, \eb, \ec, \ed, \ee, \ef),
    \end{equation}
    and
    \begin{equation}
        \label{eq:n_independent_inverse_poly}
        \norm{\tightcurlysixj{ \ya_k & \yb_k & \yc_k \\ \yd_k &  \ye_k &  \yf_k}}_{\infty} \geq \frac{1}{\polyk{6\rank[\ee]}}.
    \end{equation}
    Furthermore, for any sequence $(\ya_k, \yb_k, \yc_k, \yd_k, \ye_k, \yf_k) \in (\undomweightsp{n})^6$ converging to $(\ea, \eb, \ec, \ed, \ee, \ef) \not \in \Tetra^{+}(n)$ in the sense of \cref{eq:young_diagram_convergence}, we have for all sufficiently large $k$, the bound
    \begin{equation}
        \norm{\tightcurlysixj{ \ya_k & \yb_k & \yc_k \\ \yd_k &  \ye_k &  \yf_k}}_{\infty} \leq C e^{-kr}
    \end{equation}
    for some $C, r > 0$.
\end{restatable}
\Cref{thm:6j_estimates} is an instance of a correspondence between geometry and representation theory in the sense of geometric quantization \cite{roberts2002asymptotics,guillemin1982geometric,kirillov2004lectures}.
Indeed, the asymptotic behavior of representation theoretic objects ($6j$-symbols) is controlled by solutions of a geometric problem (the tetrahedral Horn problem).
The statement of \cref{thm:6j_estimates} should be compared with a similar statement available for $G=SU(2)$, namely the proof of the \citeauthor{ponzano1968semiclassical} asymptotic formula \cite{ponzano1968semiclassical} by \citeauthor{roberts1999classical} \cite{roberts1999classical}.
There, the exponential decay statement is similar to ours, but the polynomial bound statement is more precise for the sequence of scaled weights
\begin{equation}
    (\ya_k, \yb_k, \yc_k, \yd_k, \ye_k, \yf_k) = k (\ya, \yb, \yc, \yd, \ye, \yf).
\end{equation}
While \cref{thm:6j_estimates} is a weaker statement than the full \citeauthor{ponzano1968semiclassical} formula for $n=2$ \cite{ponzano1968semiclassical,roberts1999classical}, it nevertheless applies for all dimensions $n$.
The significant challenge in generalizing results for $n=2$ to $n > 2$ for $U(n)$ $6j$ symbols is the presence of non-trivial multiplicity spaces where $\dim C^{\yc}_{\ya \yb} > 1$.
Because of these difficulties, over the years, considerable attention has been placed on the study of specific $U(n)$ (or $U_q(\mathfrak{sl}_n)$) $6j$ symbols which happen to be multiplicity-free in the sense that the dimensions of each of the multiplicity spaces is at most one \cite{nawata2013multiplicity, alekseev2020multiplicity, alcock2023wigner, coquereaux2023integer}.
Our approach avoids the difficulties presented by these multiplicities by focusing only on the \textit{norms} of the $6j$-symbol for $U(n)$.

\subsubsection*{Methods}
Central to our proofs is Schur-Weyl duality which is used to describe the decomposition of tensor powers $X^{\ot k}$ of a matrix $X$ into irreducible representations, along with the non-commutative binomial formula for the tensor powers of matrix sums.
These features distinguish the case of $G=U(n)$ discussed in this paper from alternate versions of the tetrahedral Horn problem that one may define for other compact groups.
Related is a probabilistic technique developed in the context of quantum information theory, called \textit{eigenvalue estimation}, used for estimating the eigenvalues of non-negative matrices which model the state of quantum systems \cite{keyl2001estimating,christandl2006spectra,odonnell2016efficient,hayashi2002quantum,childs2007weak,hayashi2017group}.
Eigenvalue estimation describes how the function $\yx \mapsto \vp_\yx(X)$ is concentrated around those $\yx$ which are proportional to the eigenvalues $\ex$ of the matrix $X$ and how this concentration sharpens as $k$ grows.
The method of eigenvalue estimation explains the relation between \cref{thm:tet_iff} and \cref{thm:6j_estimates};
when $(\ya,\yb,\yc,\yd,\ye,\yf)$ are proportional to the candidate eigenvalues $(\ea,\eb,\ec,\ed,\ee,\ef)$ in the sense of \cref{eq:young_diagram_convergence}, each of the functions $\vp_\ya(a)$, $\vp_\yb(b)$, $\vp_\yc(c)$,  $\vp_\yd(d)$, $\vp_\ye(e)$, and $\vp_\yf(f)$, are comparatively large and when $(\ea,\eb,\ec,\ed,\ee,\ef) \not \in \Tetra^{+}(n)$, the $6j$-symbol is exponentially decaying and thus \cref{eq:tet_iff} fails to hold for sufficiently large $k$.

Applying the eigenvalue estimation technique to the original Horn problem was first done by \citeauthor{christandl2008quantum} \cite{christandl2008quantum}.
In this paper, we apply the eigenvalue estimation technique to the tetrahedral Horn problem.
As the method of eigenvalue estimation is quite versatile, we anticipate further works will find it applicable to the problem of determining the set of possible eigenvalues for Hermitian matrices subject to more general systems of constraints.

\subsubsection*{Related works}
Recent developments on Horn's problem include the development of the geometric method by \citeauthor{sherman2015geometric} \cite{sherman2015geometric} and \citeauthor{berline2018horn} \cite{berline2018horn} originally pioneered by the geometric proofs due to \citeauthor{belkale2002geometric} \cite{belkale2002geometric,belkale2005geometric}, as well the characterization and computation of the induced distribution over the eigenvalues $\ec$ of the sum $C=A+B$ if $A$ and $B$ are sampled uniformly and independently from the respective sets of Hermitian matrices with eigenvalues $\ea$ and $\eb$ by \citeauthor{coquereaux2019revisiting} in series of papers \cite{coquereaux2018orbital,coquereaux2018orbital,coquereaux2019revisiting,zuber2018horn}.
See also the article by \citeauthor{daftuar2005quantum} for an account of applications to topics in quantum information theory \cite{daftuar2005quantum}.

In parallel to the original Horn problem, variations and extensions of the original Horn problem have been explored by numerous authors.
For example, \citeauthor{agnihotri1997eigenvalues}, \citeauthor{belkale2008quantum}, and \cite{belkale2016multiplicative} have studied the eigenvalues of unitary matrices and their product \cite{agnihotri1997eigenvalues,belkale2008quantum,belkale2016multiplicative}, or the admissible singular values of matrices and their product \cite{speyer2005horn,thompson1971eigenvaluesviii,klyachko2000random,alekseev2001linearization}.

More recently, \citeauthor{alekseev2025multiple} have investigated a few variants of the tetrahedral Horn problem and are referred to as \textit{multiple} Horn problems in \cite{alekseev2025multiple}.
Specifically, \citeauthor{alekseev2025multiple} focus on two variants, the \textit{multiplicative multiple Horn problem} concerning the singular values of products of invertible matrices, and the \textit{tropical multiple Horn problem} concerning singular values of concatenations of weighted planar networks \cite{alekseev2025multiple}.
The tetrahedral Horn problem we consider here is termed the \textit{additive} multiple Horn problem in the language of \cite{alekseev2025multiple}.

In \cite{christandl2018recoupling}, \citeauthor{christandl2018recoupling} investigated problem of determining the eigenvalues for a non-negative Hermitian matrix $\rho_{ABC}$ on a triple tensor product space along with its partial partial traces $\rho_{AB}, \rho_{BC}, \rho_{A}, \rho_{B}$, and $\rho_{C}$.
Such eigenvalue problems arise naturally in the context of quantum information theory and are typically referred to as quantum marginal problems \cite{klyachko2004quantum, daftuar2005quantum, christandl2006spectra, christandl2007nonzero, christandl2018recoupling, yu2021complete, fraser2022sufficient}.
Using \citeauthor{nishiyama2000restriction}'s restriction theorem to the subgroup $S_k \subset U(k)$ of permutation matrices \cite{nishiyama2000restriction}, they showed how the solution to their eigenvalue problem is associated to the $6j$ symbols for the symmetric group $S_k$ \cite{christandl2018recoupling}.
Moreover, they proved the tetrahedral Horn problem may be viewed as a special case of the quantum marginals problem and, in a complementary manner, explained how norms of the $6j$ symbols for unitary group $U(n)$ are a special case of norms of the $6j$ symbols for the symmetric group $S_k$ \cite{christandl2018recoupling}, generalizing how the Littlewood-Richardson coefficients can be cast as a special case of the Kronecker coefficients \cite{littlewood1958products,murnaghan1955analysis,ikenmeyer2023all}.
In this paper, we address the tetrahedral Horn problem directly, and consequently, we obtain \cref{thm:tet_iff} which provides stronger conditions eigenvalues of tetrahedra of Hermitian matrices which go beyond the conditions obtained in \cite{christandl2018recoupling}.
Also, while the statement of \cref{thm:6j_estimates} is implicitly contained in \cite{christandl2018recoupling}, in this paper, we provide a direct proof.
We note that this direct approach was initiated, for the $n=2$ case, by \citeauthor{backens2010sixjsymbols} \cite{backens2010sixjsymbols}.

\subsubsection*{Organization of the paper}
The remainder of the paper is dedicated to the proofs of \cref{thm:tet_iff} in \cref{sec:proof_tet_iff} and \cref{thm:6j_estimates} in \cref{sec:proof_6j_estimates}.
We begin with \cref{sec:background} which provides all of the background material on representation theory and probability theory needed for the proofs which follow.
\Cref{sec:symmetries} presents various symmetries of the tetrahedral Horn problem,
\cref{sec:new_ineq_horn_orig} describes how our methods can be applied to the original Horn problem, and
\cref{sec:large_n} demonstrates how to recover an entropic inequality from the $n\to\infty$ limit.

\subsubsection*{Acknowledgments}
We are grateful to Jack Davis, Pavel Etingof, Marco Fanizza, Harold Nieuwboer, Suvrit Sra, and Michèle Vergne for helpful discussions.

Research of AA was supported in part by the grants 208235, 220040, and 236683, and by the National Center for Competence in Research (NCCR) SwissMAP of the Swiss National Science Foundation, and by the award of the Simons Foundation to the Hamilton Mathematics Institute of the Trinity College Dublin under the program ``Targeted Grants to Institutes''.
MC and TCF acknowledge financial support from the European Research Council (ERC Grant Agreement No.~818761), VILLUM FONDEN via the QMATH Centre of Excellence (Grant No.~10059) and the Novo Nordisk Foundation (grant NNF20OC0059939 `Quantum for Life'). They also thank the NCCR SwissMAP and the Section of Mathematics at the University of Geneva for their hospitality.

\section{Background material}
\label{sec:background}

In \cref{sec:background_rep_theory}, we review material on the representation theory of compact connected Lie groups,
specifically the group $U(n)$, and in \cref{sec:background_probability} we review some measure theory associated to Haar integrals of such groups.
Along the way, we introduce a series of the polynomial functions $\vp_{\yx}(X)$, $\vp^{\ya\yb}_{\yc}(X,Y)$ and $\vp^{\ya\yb\yd}_{\yc\ye\yf}(X,Y,Z)$ of one, two and three $n\times n$ matrices respectively,
and provide a formula for the representation of a matrix sum $X+Y$ which we refer to as the binomial theorem.

\subsection{Representation theory}
\label{sec:background_rep_theory}

Even though we will primarily deal with $G=U(n)$, we start with a more general story.

\subsubsection{\texorpdfstring{$G$}{G} compact}
Let $G$ be a compact connected Lie group, $T \subset G$  a maximal torus, $\mathfrak{t} = \Lie(T)$ its Lie algebra, $\Lambda = \ker(\exp) \subset \mathfrak{t}$ the integral lattice, $\Lambda^* \subset \mathfrak{t}^*$ the weight lattice, and $\Lambda^*_+ \subset \Lambda^*$ the set of dominant weights of $G$.
For $\yx \in \Lambda^*$, we denote by $\pi_\yx: G \to \End(V_\yx)$ the corresponding irreducible representation of $G$ and the corresponding character by $\chi_\yx(g) \coloneqq \Tr_{V_\yx}[\pi_\yx(g)]$.
We will always assume that the space $V_\yx$ is equipped with a Hermitian structure such that $\pi_\yx$ becomes a unitary representation.
All representations $\pi_\yx$ extend to the corresponding complex Lie group $G^\mathbb{C}$.

The tensor product of two irreducible representations admits a unique decomposition of the form
\begin{equation}
    \label{eq:tensor_product}
    V_\ya \ot V_\yb \cong \bigoplus_\yc C_{\ya \yb}^\yc \ot V_\yc,
\end{equation}
where $C_{\ya \yb}^\yc$ are multiplicity spaces with dimensions $c_{\ya \yb}^\yc = \dim  C_{\ya \yb}^\yc$.
Hermitian structures on $V_\yx$'s induces  Hermitian structures on the multiplicity spaces $C_{\ya \yb}^\yc$.
In particular, the decomposition in \cref{eq:tensor_product} is an orthogonal decomposition.
We let $\Pi_{\ya \yb}^\yc$ denote the corresponding orthogonal projection operator.
Using this notation, we have for all $g \in G$ the identity
\begin{equation}
    \label{eq:operator product}
    \pi_\ya(g) \ot \pi_\yb(g) \cong \bigoplus_{\yc} \id_{C_{\ya \yb}^\yc} \ot \pi_\yc(g),
\end{equation}
Triple tensor products can be decomposed in two alternative ways:
\begin{align}
    \label{eq:triple_tensor_product}
    \begin{split}
        (V_\ya \ot V_\yb) \ot V_\yd
        &\cong \bigoplus_{\yc, \ye} (C_{\ya \yb}^\yc \ot C_{\yc \yd}^\ye) \ot V_\ye,\\
        V_\ya \ot (V_\yb \ot V_\yd)
        &\cong \bigoplus_{\yf, \ye} (C_{\ya \yf}^\ye \ot C_{\yb \yd}^\yf)\ot V_\ye.
    \end{split}
\end{align}
We let $\projl$ and $\projr$ denote the corresponding orthogonal projection operators.
Since the composition of these two projectors is a $G$-equivariant linear map between $V_{\ye}$ and itself, Schur's Lemma implies $\projl \projr$ acts as the identity on $V_{\ye}$ and as the map
\begin{equation}
   \abcdefsymbol : C_{\ya \yb}^\yc \ot C_{\yc \yd}^\ye \to C_{\ya \yf}^\ye \ot C_{\yb \yd}^\yf,
\end{equation}
called the $6j$-symbol for $G$, on the multiplicity spaces.
Concretely, our normalization convention for $\abcdefsymbol$ is fixed by the formula
\begin{equation}
    \label{eq:prod_proj_6j}
    \projl \projr \simeq
    \abcdefsymbol \ot \id_{V_{\ye}}
\end{equation}
where $\cong$ denotes equivalence up to isometries.

Recall that for Hermitian vector spaces $V_{1}, V_{2}$ and an operator $O: V_{1} \to V_{2}$ one can introduce the norms
\begin{align}
    \begin{split}
        \norm{O}_{\infty} &= \max_{v \in V_1, \norm{v}_{V_1} = 1} \norm{O(v)}_{V_2}, \\
        \norm{O}_{2} &= \Tr_{V_1}[O^{\dagger} O]^{\frac{1}{2}} = \Tr_{V_2}[O O^{\dagger}]^{\frac{1}{2}}.
    \end{split}
\end{align}

\subsubsection{\texorpdfstring{$G=U(n)$}{G=U(n)} and  functions \texorpdfstring{$\vp_\yx$}{philambda}}

For $G=U(n)$, dominant weights are parameterized by non-increasing arrays of integers, $\yx=(\yx_1 \geq \yx_2 \geq \cdots \geq \yx_n)$, where $\yx_i \in \mathbb Z$ and $\yx_i \geq \yx_{i+1}$ for all $i$.
In particular, the first fundamental weight $\varpi_1=(1, 0, \ldots, 0)$ corresponds to the defining representation of $U(n)$ on $V_{\varpi_1} = \mathbb{C}^n$.

We say that a dominant weight $\yx$ is positive if each of its components is positive (i.e., $\yx_n \geq 0$), and we denote the set of positive dominant weights by
\begin{equation}
    \undomweightsp{n} \coloneqq \{ \yx=(\yx_1, \yx_2, \ldots, \yx_n) \mid \yx_i \in \mathbb Z, \yx_i \geq \yx_{i+1}, \yx_n \geq 0. \}
\end{equation}
For positive dominant weights, we let
\begin{equation}
    \syx \coloneqq \sum_{i=1}^n \yx_i,
    \qquad
    \text{and}
    \qquad
    \ell(\yx) \coloneqq \card{\{i \mid \yx_i \neq 0\}},
\end{equation}
where $\syx$ is called the size of $\yx$ and $\ell(\yx)$ is called the depth of $\yx$.
Henceforth we use the notation $\yx \vdash_{n} k$ to denote the set of positive dominant weights with size $k$ and length at most $n$, and $\yx \vdash k$ to denote the set of positive dominant weights with size $k$ and arbitrary length.
Positive dominant weights, $\yx \in \undomweightsp{n}$, are equivalent to integer partitions of $\syx$ into $\ell(\yx) \leq n$ parts and depicted as Young diagrams with $\syx$ boxes and $\ell(\yx)$ rows.
Representations $\pi_\yx(g)$ corresponding to positive dominant weights are polynomial in matrix elements of $g$, and thus they extend to matrix algebras $\Mat_n(\mathbb{C})$ such that
\begin{equation}
    \pi_\yx(X) \pi_\yx(Y) = \pi_\yx(XY)
\end{equation}
for all $X,Y \in \Mat_n(\mathbb{C})$.
If $X$ is a Hermitian matrix with eigenvalues $\ex = (\ex_1 \geq \cdots \geq \ex_n)$ then the character of $X$ for the representation $\pi_{\yx}$, denoted by $\chi_{\yx}(X)$, is equal to the Schur-polynomial of its eigenvalues $\ex$, denoted by $s_{\yx}(\ex)$.
Moreover, the dimension for the irreducible representation $V_{\yx}$ is given by Weyl's dimension formula for $U(n)$ \cite{hall2008representations,itzykson1966unitary},
\begin{equation}
    \label{eq:weyl_dim}
    \dim(V_{\yx}) = \frac{\prod_{1 \leq i<j \leq n}(\yx_i - i - \yx_j + j)}{\prod_{1 \leq i<j \leq n}(j-i)}.
\end{equation}
Tensor products of representations with positive dominant weights decompose into direct sums of representations of the same type.
Note that in the decomposition of the tensor product representation $V_{\ya}\ot V_{\yb}$ appearing in \cref{eq:tensor_product}, the only positive dominant weights $\syc$ which appear are those which satisfy the size condition
\begin{equation}
    \label{eq:size_condition_tri}
    \syc = \sya + \syb.
\end{equation}

Recall that Young diagrams, $\yx$, with fixed size $\syx=k$ and at most $\ell(\yx) \leq k$ rows also label isomorphism classes of irreducible representations, $W_\yx$, of the symmetric group $S_k$.
The dimensions of these representations are given by the formula
\begin{equation}
    \label{eq:dim_W_hook_product}
    \dim(W_\yx) = \frac{\syx!}{H_{\yx}},
    \qquad
    \text{with}
    \qquad
    H_{\yx} = \prod_{b \in \yx} h_\yx(b),
\end{equation}
where $h_\yx(b)$ is the hook length of box $b$ and is defined as the number of boxes to the right and below $b$ plus one.
We refer to $H_{\yx}$ as the hook product of $\yx$.
For the future use, we introduce rescaled representations
\begin{equation}
    \Phi_\yx(X) = \frac{1}{H_{\yx}} \pi_\yx(X),
\end{equation}
and rescaled characters
\begin{equation}
    \vp_\yx(X) = \Tr_{V_\yx}  \Phi_\yx(X) = \frac{1}{H_{\yx}}  \chi_\yx(X).
\end{equation}
This rescaling of characters produces the normalization conditions
\begin{equation}
    \label{eq:varphi_normalization}
    \sum_{\yx \vdash k} \vp_\yx(X) = \frac{(\Tr X)^k}{k!},
    \quad \text{and} \quad
    \sum_\yx \vp_\yx(X) = \exp(\Tr X).
\end{equation}
Indeed, by the Schur-Weyl duality, there is an isomorphism of $S_k \times U(n)$-modules
\begin{equation}
    \label{eq:schur-weyl}
    (\mathbb{C}^n)^{\ot k} \cong \bigoplus_{\yx \vdash_{n} k} W_\yx \ot V_\yx,
\end{equation}
where the symmetric group $S_k$ acts by permutations of tensor factors in $(\mathbb{C}^n)^{\ot k}$.
For each Young diagram $\yx$ with $k$ boxes and at most $n$ rows, we let $\Pi^{\yx}$
denote the unique $S_k \times U(n)$-equivariant orthogonal projection $(\mathbb{C}^n)^{\ot k}$ into $W_{\yx} \ot V_{\yx}$.
Since all $U(n)$ representations entering the decomposition in \cref{eq:schur-weyl} are positive, the Schur-Weyl duality extends to an isomorphism of modules for monoids $S_k \times \Mat_n(\mathbb{C})$.
Applying this isomorphism to an element $X \in \Mat_n(\mathbb{C})$, we obtain
\begin{equation}
    \label{eq:X^k}
    X^{\ot k} \cong \bigoplus_{\yx \vdash_{n} k} \id_{W_\yx} \ot \pi_\yx(X).
\end{equation}
Taking a trace of this equation over $(\mathbb{C}^n)^{\ot k}$ yields the normalization condition in \cref{eq:varphi_normalization}.

Recall the space of $n\times n$ Hermitian matrices is denoted by $\herm{n}$ and the subset of non-negative $n\times n$ Hermitian matrices by $\hermp{n}$ and given a Hermitian matrix $X \in \herm{n}$, $\eig X$ denotes the ordered tuple of eigenvalues of $X$.
\begin{prop}
    For $X \in \hermp{n}, \yx \in \undomweightsp{n}$, the operator $\Phi_\yx(X)$ is Hermitian and non-negative, and $\vp_\yx(X)$ is a non-negative real value.
\end{prop}
\begin{proof}
    By \cref{eq:X^k}, the operators $\id_{W_\yx} \ot \pi_\yx(X)$ are orthogonal restrictions of the non-negative Hermitian operator $X^{\ot k}$.
    Hence, they are also non-negative Hermitian operators, and thus so are $\pi_\yx(X)$ and its rescaling $\Phi_\yx(X)$.
    Since $\vp_\yx(X)$ is the trace of $\Phi_\yx(X)$, it takes values in $[0, \infty)$.
\end{proof}

\subsubsection{The function \texorpdfstring{$\vp^{\ya \yb}_\yc$}{phiabc} and the binomial theorem}

It is  convenient to introduce the following functions of two matrices,
\begin{equation}
    \vp^{\ya \yb}_\yc(X, Y) =
    \Tr_{V_\ya \ot V_\yb}  (\Phi_\ya(X) \ot \Phi_\yb(Y)) \Pi_{\ya\yb}^\yc.
\end{equation}

\begin{prop}
    For $\ya, \yb \in \undomweightsp{n}$ and $X,Y \in \Mat_n(\mathbb{C})$, we have
    \begin{equation}
        \label{eq:varphi_product_sum}
        \vp_\ya(X) \vp_\yb(Y) = \sum_\yc \vp^{\ya \yb}_\yc(X, Y).
    \end{equation}
    For $X,Y \in \hermp{n}$, the function $\vp^{\ya \yb}_\yc(X, Y)$ takes non-negative real values.
\end{prop}

\begin{proof}
    To prove the summation identity, note that
    \begin{equation}
        \sum_{\yc} \Pi_{\ya\yb}^\yc = \id_{V_{\ya}} \ot \id_{V_{\yb}},
    \end{equation}
    which then implies
    \begin{align}
        \begin{split}
            \sum_\yc \vp^{\ya \yb}_\yc(X, Y)
            &= \sum_\yc \Tr_{V_\ya \ot V_\yb}  (\Phi_\ya(X) \ot \Phi_\ya(Y))\Pi_{\ya\yb}^\yc, \\
            &= \Tr_{V_\ya \ot V_\yb} (\Phi_\ya(X) \ot \Phi_\ya(Y)), \\
            &= \vp_\ya(X) \vp_\yb(Y).
        \end{split}
    \end{align}
    For the claim of positivity, note that $\Pi_{\ya\yb}^\yc$ is a orthogonal projection operator which implies $\vp^{\ya \yb}_\yc(X, Y)$ is the trace of the orthogonal restriction of the non-negative Hermitian operator $(\Phi_\ya(X) \ot \Phi_\ya(Y))$ over the subspace $(C_{\ya \yb} \ot V_\yc) \subseteq V_\ya \ot V_\yb$, meaning $\vp^{\ya \yb}_\yc(X, Y) \geq 0$ as claimed.
\end{proof}

\begin{thm}[Binomial theorem]
    \label{thm:the_binomial_theorem}
    For $X,Y \in \Mat_n(\mathbb{C})$ and $\yc \in \undomweightsp{n}$, we have
    \begin{equation}
        \label{eq:varphi_sum_sum}
        \vp_\yc(X+Y)=\sum_{\ya, \yb} \vp^{\ya \yb}_\yc(X, Y).
    \end{equation}
\end{thm}
\begin{proof}
    Let $k=\syc$ and let $\rho_k$ denote the representation of the monoid $(S_k \times \Mat_n(\mathbb{C}))$ on the space $(\mathbb{C}^n)^{\ot k}$, and $\Pi^\yc$ the orthogonal projector to the summand $W_\yc \ot V_\yc$ in the Schur-Weyl decomposition \cref{eq:schur-weyl}.
    Recall that $\dim(W_{\yx}) = \syx! / H_{\yx}$.
    Then we have for all $k$ and $X \in \Mat_n(\mathbb{C}))$,
    \begin{align}
        \begin{split}
            \vp_{\yx}(X)
            = \frac{\dim(W_{\yx})}{k!} \Tr_{V_\yx}[\pi_\yx(X)]
            = \frac{1}{k!} \Tr_{(\mathbb{C}^n)^{\ot k}}[\Pi_{\yx} \rho_{k}(X)].
        \end{split}
    \end{align}
    Moreover, we have
    \begin{equation}
        (X+Y)^{\ot k} = \sum_{l=0}^k \frac{1}{l!(k-l)!}\sum_{\sigma \in S_k} \rho_k(\sigma) (X^{\ot l} \ot Y^{\ot (k-l)})\rho_k(\sigma^{-1})
    \end{equation}
    which is a non-commutative counterpart of the binomial decomposition.
    We then expand
    \begin{align}
        \begin{split}
            &\vp_\yc(X+Y) \\
            &\quad= \frac{1}{k!} \Tr_{(\mathbb{C}^n)^{\ot k}} \Pi^\yc  (X+Y)^{\ot k}, \\
            &\quad= \frac{1}{k!}\sum_{l=0}^k \sum_{\sigma \in S_k} \frac{1}{l! (k-l)!} \Tr_{(\mathbb{C}^n)^{\ot k}} \Pi^\yc \rho_k(\sigma) (X^{\ot l} \ot Y^{\ot (k-l)})\rho_k(\sigma^{-1}).
        \end{split}
    \end{align}
    Here we make use of cyclicity of the trace and the $S_k$-equivariance of $\Pi^{\yc}$, i.e., $\rho_k(\sigma^{-1}) \Pi^\yc \rho_k(\sigma) = \Pi^\yc$ to obtain
    \begin{align}
        \begin{split}
            \vp_\yc(X+Y)
            = \sum_{l=0}^k \frac{1}{l! (k-l)!} \Tr_{(\mathbb{C}^n)^{\ot k}} \Pi^\yc (X^{\ot l} \ot Y^{\ot (k-l)}).
        \end{split}
    \end{align}
    Now we freely insert operator identities $\sum_{\ya \vdash k-l} \Pi^{\ya} = \id_{\mathbb C^n}^{\otimes k-l}$ and $\sum_{\yb \vdash l} \Pi^{\yb} = \id_{\mathbb C^n}^{\otimes l}$ to obtain the claimed result:
    \begin{align}
        \begin{split}
            \vp_\yc(X+Y)
            &= \sum_{\ya, \yb}  \frac{1}{\sya! \syb!} \Tr_{(\mathbb{C}^n)^{\ot k}} \Pi^\yc (\Pi^{\ya}X^{\ot l} \ot \Pi^{\yb} Y^{\ot (k-l)}), \\
            &= \sum_{\ya, \yb}  \frac{1}{H_\ya H_\yb} \Tr_{V_\ya \ot V_\yb} \Pi_{\ya \yb}^\yc  (\pi_\ya(X) \ot \pi_\yb(Y)), \\
            &= \sum_{\ya, \yb}  \Tr_{V_\ya \ot V_\yb} \Pi_{\ya \yb}^\yc  (\Phi_\ya(X) \ot \Phi_\yb(Y)), \\
            &= \sum_{\ya, \yb}  \yf_{\ya \yb}^\yc(X,Y).
        \end{split}
    \end{align}
\end{proof}

\subsubsection{The function \texorpdfstring{$\vp^{\ya \yb \yd}_{\yc \ye \yf}$}{phiabcdef}}

For a given tuple of positive dominant weights $(\ya,\yb,\yc,\yd,\ye,\yf) \in \undomweightsp{n}$, we introduce a function of three matrices:
\begin{equation}
    \vp^{\ya \yb \yd}_{\yc \ye \yf}(X, Y, Z)
    \coloneqq
    \Tr_{V_\ya \ot V_\yb \ot V_\yd}
    \projr (\Phi_\ya(X) \ot \Phi_\yb(Y) \ot \Phi_\yd(Z))\projl.
\end{equation}
Note that $\vp^{\ya \yb \yd}_{\yc \ye \yf}(X,Y,Z) = 0$ whenever $\abcdefsymbol = 0$.
Even when restricted to $(\hermp{n})^{\times 3}$, functions $\vp^{\ya \yb \yd}_{\yc \ye \yf}$ take complex values, in general.
However, the following is true:
\begin{prop}
    \label{prop:non_neg_partial}
    For $X,Y,Z \in \hermp{n}$, the functions
    \begin{align}
        \begin{split}
            \vp^{\ya\yb\yd}_{\_ \ye\yf}(X,Y,Z) & = \sum_{\yc}\vp^{\ya\yb\yd}_{\yc\ye\yf}(X,Y,Z), \\
            \vp^{\ya\yb\yd}_{\yc\ye \_}(X,Y,Z) & = \sum_{\yf}\vp^{\ya\yb\yd}_{\yc\ye\yf}(X,Y,Z)
        \end{split}
    \end{align}
    take non-negative real values.
\end{prop}
\begin{proof}
    From the definition of $\vp^{\ya \yb \yd}_{\yc \ye \yf}(X, Y, Z)$ and the identity $\sum_\yc \projl = \Pi^\ye_{\ya \yb \yd}$, where $\Pi^\ye_{\ya \yb \yd}$ is the orthogonal projector to the isotypical component of $\ye$ in the triple tensor product, we obtain
    \begin{align}
        \begin{split}
            \vp^{\ya\yb\yd}_{\_ \ye\yf}(X,Y,Z)
            =
            \Tr_{V_\ya \ot V_\yb \ot V_\yd} \projr (\Phi_\ya(X) \ot \Phi_\yb(Y) \ot \Phi_\yd(Z))
        \end{split}
    \end{align}
    because $\Pi^\ye_{\ya \yb \yd} \projr = \projr$.
    The right hand side in the equation above is the trace of the orthogonal restriction of the non-negative Hermitian operator $(\Phi_\ya(X) \ot \Phi_\yb(Y) \ot \Phi_\yd(Z))$ to the image of the projector $\projr$, hence $\vp^{\ya\yb\yd}_{\_ \ye\yf}(X,Y,Z) \geq 0$.
    The proof for $\vp^{\ya\yb\yd}_{\yc\ye \_}(X,Y,Z)$ is analogous.
\end{proof}

\subsubsection{Estimates for \texorpdfstring{$\vp^{\ya \yb \yd}_{\yc \ye \yf}$}{phiabcdef}}

Here we derive bounds on the magnitude of the quantity $\vp^{\ya \yb \yd}_{\yc \ye \yf}(X, Y, Z)$.
To begin we need the following proposition.
\begin{prop}\label{prop:cauchy_schwarz_ineq_double_proj}
    Let $P, Q$ be two orthogonal projections, and $\Omega$ be a non-negative Hermitian operator.
    Then,
    \begin{equation}
        \label{eq:cauchy_schwarz_ineq_double_proj}
        \abs{\Tr(Q \Omega P)} \leq \norm{PQ}_{\infty} (\Tr Q\Omega)^{\frac{1}{2}}(\Tr P\Omega )^{\frac{1}{2}}.
    \end{equation}
\end{prop}
\begin{proof}
    Using the Cauchy-Schwarz inequality for the Hilbert-Schmidt norm, we obtain
    \begin{align}
        \begin{split}
            \abs{\Tr(Q \Omega P)}
            &= \abs{\Tr(P Q \Omega P)}, \\
            &= \abs{\Tr(P Q \Omega^{\frac{1}{2}}) (\Omega^{\frac{1}{2}} P)}, \\
            &\leq (\Tr P Q \Omega Q P)^{\frac{1}{2}} (\Tr P \Omega P)^{\frac{1}{2}}, \\
            &= (\Tr P Q \Omega Q)^{\frac{1}{2}} (\Tr P \Omega)^{\frac{1}{2}}.
        \end{split}
    \end{align}
    Now let $\{ v_j \}_{j=1}^{n}$ denote a set of orthonormal eigenvectors for the operator $Q \Omega Q$ with eigenvalues $\omega_j \geq 0$.
    Then,
    \begin{align}
        \begin{split}
            \Tr(P Q \Omega Q)
            &= \Tr (Q P^2 Q)(Q \Omega Q), \\
            &= \sum_j \omega_j \norm{P Q v_j}^2, \\
            &\leq \norm{P Q}_\infty^2  \sum_{j} \omega_j, \\
            &= \norm{P Q}_\infty^2 \Tr(Q \Omega).
        \end{split}
    \end{align}
    This completes the proof.
\end{proof}
Applying this proposition to the definition of $\vp^{\ya\yb\yd}_{\yc\ye\yf}(X,Y,Z)$ we conclude the following.
\begin{prop}
    \label{prop:varphi_tet_main_ineq}
    Let $X,Y,Z \in \hermp{n}$. Then,
    \begin{equation}
        \abs{\vp^{\ya\yb\yd}_{\yc\ye\yf}(X,Y,Z)} \leq \norm{\abcdefsymbol}_{\infty} \vp^{\ya\yb\yd}_{\_\ye\yf}(X,Y,Z)^{\frac{1}{2}} \vp^{\ya\yb\yd}_{\yc\ye\_}(X,Y,Z)^{\frac{1}{2}}.
    \end{equation}
\end{prop}
\begin{proof}
    Let $\Omega_{\ya\yb\yd}^{X,Y,Z} \coloneqq \Phi_{\ya}(X)\ot\Phi_{\yb}(Y)\ot\Phi_{\yd}(Z)$.
    Under the assumption that $X,Y,Z \geq 0$, we have that $\Omega_{\ya\yb\yd}^{X,Y,Z} \geq 0$.
    Therefore, \cref{prop:cauchy_schwarz_ineq_double_proj} implies
    \begin{align}
        \begin{split}
            &\abs{\Tr[\projr\Omega_{\ya\yb\yd}^{X,Y,Z}\projl]} \\
            &\qquad\leq \norm{\projl \projr}_{\infty} \Tr[\projl \Omega_{\ya\yb\yd}^{X,Y,Z}]^{\frac{1}{2}}\Tr[\projr \Omega_{\ya\yb\yd}^{X,Y,Z}]^{\frac{1}{2}}
        \end{split}
    \end{align}
    which is equivalent to the claim due to \cref{eq:prod_proj_6j}.
\end{proof}
Our next proposition enables us to place bounds on $\vp^{\ya\yb\yd}_{\_\ye\yf}(X,Y,Z)$ and $\vp^{\ya\yb\yd}_{\yc\ye\_}(X,Y,Z)$ which depend only the eigenvalues of $X,Y,X+Y,Z,X+Y+Z$, and $Y+Z$ and not their eigenvectors.
These bounds, in conjunction with \cref{prop:varphi_tet_main_ineq} will help us in the proof of \cref{thm:tet_iff} later.
\begin{prop}
    \label{prop:tet_psum_ineq}
    Let $X,Y,Z \in \hermp{n}$.
    Then,
    \begin{align}
        \label{eq:psum_prod_bounds}
        \begin{split}
            \vp^{\ya\yb\yd}_{\_\ye\yf}(X,Y,Z) &\leq \vp_{\ya}(X)\vp_{\yf}(Y+Z),\\
            \vp^{\ya\yb\yd}_{\_\ye\yf}(X,Y,Z) &\leq \vp_{\ya}(X)\vp_{\yb}(Y)\vp_{\yd}(Z),\\
            \vp^{\ya\yb\yd}_{\_\ye\yf}(X,Y,Z) &\leq \vp_{\ye}(X+Y+Z),
        \end{split}
        \begin{split}
            \vp^{\ya\yb\yd}_{\yc\ye\_}(X,Y,Z) &\leq \vp_{\yc}(X+Y)\vp_{\yd}(Z),\\
            \vp^{\ya\yb\yd}_{\yc\ye\_}(X,Y,Z) &\leq \vp_{\ya}(X)\vp_{\yb}(Y)\vp_{\yd}(Z),\\
            \vp^{\ya\yb\yd}_{\yc\ye\_}(X,Y,Z) &\leq \vp_{\ye}(X+Y+Z).
        \end{split}
    \end{align}
\end{prop}
\begin{proof}
    Each of these bounds are consequences of \cref{thm:the_binomial_theorem} and \cref{prop:non_neg_partial}.
    Summing $\vp^{\ya\yb\yd}_{\_\ye\yf}(X,Y,Z)$ over $\ye$ yields
    \begin{align}
        \label{eq:psum_yf_over_ye}
        \begin{split}
            \sum_{\ye} \vp^{\ya\yb\yd}_{\_\ye\yf}(X,Y,Z)
            &= \sum_{\ye} \Tr_{V_\ya \ot V_\yb \ot V_\yd} \projr (\Phi_\ya(X) \ot \Phi_\yb(Y) \ot \Phi_\yd(Z)), \\
            &= \Tr_{V_\ya \ot V_\yb \ot V_\yd} (\id_{V_{\ya}} \ot \Pi^{\yf}_{\yb\yd}) (\Phi_\ya(X) \ot \Phi_\yb(Y) \ot \Phi_\yd(Z)), \\
            &= \Tr_{V_\ya} \Phi_\ya(X) \Tr_{V_\yb \ot V_\yd} \Pi^{\yf}_{\yb\yd} (\Phi_\yb(Y) \ot \Phi_\yd(Z)), \\
            &= \vp_{\ya}(X)\vp_{\yf}^{\yb\yd}(Y,Z).
        \end{split}
    \end{align}
    Then, summing the result over $\yb$ and $\yd$ and applying \cref{thm:the_binomial_theorem} produces
    \begin{align}
        \begin{split}
            \sum_{\ye\yb\yd} \vp^{\ya\yb\yd}_{\_\ye\yf}(X,Y,Z)
            &= \vp_{\ya}(X)\sum_{\yb\yd}\vp_{\yf}^{\yb\yd}(Y,Z)
            = \vp_{\ya}(X)\vp_{\yf}(Y+Z).
        \end{split}
    \end{align}
    Now using the non-negativity of $\vp^{\ya\yb\yd}_{\_\ye\yf}(X,Y,Z)$, we obtain the first claimed inequality in \cref{eq:psum_prod_bounds}.
    Similarly, the second claimed inequality in \cref{eq:psum_prod_bounds} follows analogously from summing
    $\vp^{\ya\yb\yd}_{\_\ye\yf}(X,Y,Z)$ over $\ye$ and $\yf$,
    while the third claimed inequality in \cref{eq:psum_prod_bounds} follows from summing $\vp^{\ya\yb\yd}_{\_\ye\yf}(X,Y,Z)$ over $\yf$, $\ya$, $\yb$ and $\yd$.
    The inequalities concerning $\vp^{\ya\yb\yd}_{\yc\ye\_}(X,Y,Z)$ are proven in an analogous manner.
\end{proof}

\subsection{Probability measures}
\label{sec:background_probability}

In this section, we collect information about probability measures on finite sets and on conjugacy classes of Hermitian matrices needed later in the paper.

\subsubsection{Measures on finite sets}

Recall that for a finite set $[r]=\{ 1, \ldots, r\}$ a probability measure is defined by an array of non-negative real numbers $(p_1, \ldots, p_r)$ such that $\sum_{i=1}^r p_i =1$.
For two measures $p$ and $q$ on $[r]$, the Kolmogorov distance $K(p,q)$ between them is defined by:
\begin{equation}
    K(p,q) \coloneqq \frac{1}{2}\norm{p-q}_1
    = \frac{1}{2} \sum_{i=1}^r \abs{p_i - q_i}.
\end{equation}
Another measure of similarity between distributions $p$ and $q$ is the Bhattacharyya coefficient \cite{bhattacharyya1946measure,fuchs2002cryptographic}:
\begin{equation}
    \label{eq:Bhattacharyya_coefficient}
    \BC(p,q) = \sum_{i=1}^r \sqrt{p_iq_i}.
\end{equation}
For all distributions $p, q$, the Bhattacharyya coefficient $\BC(p,q)$ lies in the interval $[0,1]$, and $\BC(p,q)=1$ if and only if $p=q$.
Recall the following relation between
the Kolmogorov distance and the Bhattacharyya coefficient (see {\em e.g.} \cite[Proposition 5, Equation (41)]{fuchs2002cryptographic}):
\begin{equation}
    1-\BC(p,q) \leq K(p,q) \leq \sqrt{1-\BC(p,q)^2}.
\end{equation}
This implies\footnote{Here we used the inequality $\sqrt{1-x^2} \leq \exp(-x^2/2)$ which holds for all $x \in [0,1]$.}
\begin{equation}
    \label{eq:BC_norm_estimate}
    \BC(p,q) \leq \sqrt{1-K(p,q)^2} \leq \exp\left(- \frac{1}{2} K(p,q)^2\right)=
    \exp\left(- \frac{1}{8} \norm{p-q}_1^2\right).
\end{equation}
We extend the Bhattacharyya coefficient to non-normalized non-negative tuples where $\sum_i p_i = \sum_i q_i = \tau >0$.
In that case, the \textit{unnormalized} Bhattacharyya coefficient satisfies
\begin{equation}
    \BC(p,q) = \tau  \BC\left(\frac{p}{\tau}, \frac{q}{\tau}\right),
\end{equation}
and the estimate in \cref{eq:BC_norm_estimate} then reads
\begin{equation}
    \label{eq:bhattacharyya_vs_distance}
    \frac{\BC(p,q)}{\tau} \leq \exp\left(-\frac{1}{8\tau^2}\norm{p-q}_1^2\right).
\end{equation}
Finally, the Kullback-Leibler (KL) divergence is defined as
\begin{equation}
    \label{eq:KL_divergence}
    \kl{p}{q} = \sum_{i=1}^{n} p_i \ln \frac{p_i}{q_i}.
\end{equation}
It takes values between $0$ and $+\infty$, and it is equal to zero if and only if $p=q$.

\subsubsection{Measures on partitions}

Recall the set of eigenvalues of $n\times n$ non-negative Hermitian matrices is noted by $\eigsp{n}$.
One can associate to each $\ex \in \eigsp{n}$ a probability measure $p^\ex$ on the set $[n]$ by normalization:
\begin{equation}
    p^\ex_i = \frac{\ex_i}{\tex}, \quad \text{where} \quad \tex = \sum_{i=1}^n \ex_i.
\end{equation}
For each $\ex$ and each $k \geq 1$, using \cref{eq:varphi_normalization}, one can define a probability measure on the set of positive dominant weights of $U(n)$ with size $\syx = k$:
\begin{equation}
    \label{eq:schur_weyl_dist}
    \SW{\ex}(\yx) \coloneqq \frac{k!}{\tex^k} \vp_\yx(x) = \frac{\Tr[\Pi_{\yx} X^{\ot k}]}{\Tr[X^{\ot k}]},
\end{equation}
where $\vp_\yx(x) \coloneqq \vp_{\yx}(\diag(\ex_1, \ldots, \ex_n))$.
This probability distribution is known as the \textit{Schur-Weyl} distribution for $x$ of degree $k = \syx$ \cite{odonnell2016efficient}.
A key property of the Schur-Weyl distribution $\SW{\ex}(\yx)$ is that it behaves similarly to the multinomial distribution $m_{\ex}(\yx)$ where
\begin{equation}
    m_{\ex}(\yx) \coloneqq \frac{\syx!}{\tex^{k}}\frac{\ex_1^{\yx_1}}{\yx_1!}\frac{\ex_2^{\yx_2}}{\yx_2!} \cdots \frac{\ex_n^{\yx_n}}{\yx_n!}.
\end{equation}
\begin{prop}
    \label{prop:multinomial_bound}
    Let $\yx \in \undomweightsp{n}$ and $\ex \in \eigsp{n}$.
    Then
    \begin{equation}
        \SW{\ex}(\yx) \leq \dim V_{\yx} m_{\ex}(\yx).
    \end{equation}
\end{prop}
\begin{proof}
    The proof follows from (i) an upper bound on the Schur-polynomial $s_{\yx}(\ex)$,
    \begin{equation}
        s_{\yx}(\ex) \leq \dim V_{\yx} \ex_1^{\yx_1}\ex_2^{\yx_2}\cdots \ex_n^{\yx_n},
    \end{equation}
    which follows because $s_{\yx}(\ex)$ can be expressed as a sum over $\dim V_{\yx}$ multinomials of the form $\ex_1^{\mu_1}\ex_2^{\mu_2}\cdots \ex_n^{\mu_n}$ with $\ex_1^{\yx_1}\ex_2^{\yx_2}\cdots \ex_n^{\yx_n}$ the largest, and (ii) a lower bound on the hook product,
    \begin{equation}
        H_{\yx} \geq \yx_1!\cdots \yx_n!,
    \end{equation}
    which follows because $h_{\yx}(b) \geq \yx_i - j + 1$ holds for the box $b$ in row $i$ and column $j$.
\end{proof}
\begin{rem}
    We will henceforth relax the dimension prefactor $\dim V_{\yx}$ using the following dimension bound which holds for all Young diagrams with $k$ boxes:
    \begin{equation}
        \label{eq:unitary_irrep_dim_bound}
        \dim V_{\yx} \leq \polyk{\binom{n}{2}}.
    \end{equation}
\end{rem}
The following result is known as eigenvalue (or spectrum) estimation \cite{alicki1988symmetry,keyl2001estimating,hayashi2002quantum,odonnell2016efficient}.
The following result is \cite[Theorem 1]{christandl2006spectra}.
\begin{prop}[Eigenvalue estimation]
    \label{prop:eig_est}
    Let $x \in \eigsp{n}, \yx \in \undomweightsp{n}$ with $\syx = k$.
    Then the Schur-Weyl distribution $\SW{\ex}(\yx)$ satisfies
    \begin{equation}
        \SW{\ex}(\yx) \leq \polyk{\binom{n}{2}} \exp\left[-k \kl{p^{\yx}}{p^\ex}\right].
    \end{equation}
\end{prop}
\begin{proof}
    From \cref{prop:multinomial_bound} we have the bound $\SW{\ex}(\yx) \leq \polyk{\binom{n}{2}} m_{\ex}(\yx)$.
    What remains is to bound the multinomial distribution $m_{\ex}(\yx)$ using the Kullback-Leibler divergence.
    To this end, note that $k = \syx = \yx_1 + \cdots + \yd_n$ which implies
    \begin{align}
        \begin{split}
            m_{\ex}(\yx)
            &= \frac{k!}{\tex^{k}}\frac{\ex_1^{\yx_1}}{\yx_1!}\cdots \frac{\ex_n^{\yx_n}}{\yx_n!}, \\
            &= \left(\frac{k!}{k^{k}} \frac{\yx_1^{\yx_1}}{\yx_1!} \cdots \frac{\yx_n^{\yx_n}}{\yx_n!}\right)
            \left(\frac{k^{k}}{\tex^{k}}\frac{\ex_1^{\yx_1}}{\yx_1^{\yx_1}}\cdots \frac{\ex_n^{\yx_n}}{\yx_n^{\yx_n}}\right), \\
            &= m_{\yx}(\yx)
            \exp\left[ - k \sum_{i=1}^{n} \frac{\yx_i}{k} \ln \left(\frac{\yx_i}{k} \frac{\tex}{\ex_i}\right)  \right], \\
            &= m_{\yx}(\yx)
            \exp\left[ - k \kl{\frac{\yx}{k}}{\frac{\ex}{\tex}} \right].
        \end{split}
    \end{align}
    Since $m_{\yx}(\yx) \in [0,1]$ is a probability, we have the bound $m_{\yx}(\yx) \leq 1$ and therefore we obtain the claimed result:
    \begin{equation}
        \SW{\ex}(\yx) \leq \polyk{\binom{n}{2}} m_{\ex}(\yx) \leq \polyk{\binom{n}{2}}\exp\left[ - k \kl{\frac{\yx}{k}}{\frac{\ex}{\tex}} \right].
    \end{equation}
\end{proof}
In summary, the eigenvalue estimation result captures the fact that, like the multinomial distribution $m_x(\yx)$, the Schur-Weyl distribution $\SW{\ex}(\yx)$ for large $k$ is sharply concentrated around Young diagrams $\yx$ whose distribution, $p^{\yx} = \yx/\syx$, is similar to the eigenvalue distribution $p^{\ex} = \ex/\tex$.

The following proposition compares Schur-Weyl measures $\SW{\ex}$ and $\SW{\ey}$ for different $\ex, \ey \in \eigsp{n}$.
\begin{prop}[Eigenvalue separation]
    \label{prop:eig_sep}
    Let $\ex, \ey \in \mathbb R_n^{+}$ satisfy $\tau = \tex = \tey$.
    For any non-negative integer $k\geq 1$,
    \begin{equation}
        \label{eq:eig_sep}
        \BC(\SW{\ex}, \SW{\ey}) \leq \polyk{\binom{n}{2}} \BC(p^\ex,p^\ey)^{k}.
    \end{equation}
\end{prop}
\begin{proof}
    From \cref{prop:multinomial_bound} and \cref{eq:unitary_irrep_dim_bound}, we have
    \begin{align}
        \begin{split}
            \BC(\SW{\ex},\SW{\ey})
            &= \sum_{\yx \vdash k} \sqrt{\SW{\ex}(\yx)\SW{\ey}(\yx)}, \\
            &\leq \sum_{\yx \vdash k} \dim V_{\yx} \sqrt{m_\ex(\yx)m_\ey(\yx)}, \\
            &\leq \polyk{\binom{n}{2}} \sum_{\yx \vdash k}\sqrt{m_\ex(\yx)m_\ey(\yx)}, \\
            &= \polyk{\binom{n}{2}} \BC_k(m_\ex,m_\ey).
        \end{split}
    \end{align}
    where $\BC_k(m_\ex,m_\ey)$ is the Bhattacharyya coefficient between the multinomial distributions $m_{\ex}$ and $m_{\ey}$.
    Since Bhattacharyya coefficients are multiplicative with respect to multinomial distributions, i.e.,
    \begin{equation}
        \BC_k(m_x, m_y) = \BC\left(p^{\ex}, p^{\ey}\right)^{k},
    \end{equation}
    we conclude \cref{eq:eig_sep}.
\end{proof}
In subsequent sections we will frequently use a \textit{weaker} approximation which holds for all Young diagrams $\yx$ of size $k$:
\begin{equation}
    \sqrt{\SW{\ex}(\yx)\SW{\ey}(\yx)} \leq \polyk{\binom{n}{2}} \BC(p^{\ex},p^{\ey})^{k},
\end{equation}
which follows from \cref{prop:eig_sep} because $\sqrt{\SW{\ex}(\yx)\SW{\ey}(\yx)} \leq \BC(\SW{\ex}, \SW{\ey})$ by definition.
Restating this bound in terms of $\varphi_{\ya}(\ex) = \frac{\tex^{k}}{k!}\SW{\ex}(\yx)$, we observe
\begin{equation}
    \label{eq:weak_eig_sep_varphi}
    \sqrt{\vp_{\ex}(\yx)\vp_{\ey}(\yx)} \leq \polyk{\binom{n}{2}} \frac{\BC(\ex,\ey)^{\syx}}{\syx!},
\end{equation}
where on the right-hand side we use the unnormalized Bhattacharyya coefficient for notational convenience.

\subsubsection{Measures on conjugacy classes}
For $x \in \eigsp{n}$, we define the conjugacy class $\orbit_x$ of the diagonal matrix
$\diag(x_1, \ldots, x_n)$ under the action of the unitary group $U(n)$:
\begin{equation}
    \orbit_x = \{ X = U\diag(x_1, \ldots, x_n)U^{-1} \mid U \in U(n)\} \subset \hermp{n}.
\end{equation}
Each conjugacy class carries a unique $U(n)$-invariant probability measure $\mu_x$.
This measure can be constructed as the push-forward of the Haar probability measure on $U(n)$.
We write $X \sim \orbit_x$ to indicate $X$ distributed according to the measure $\mu_x$.

\begin{prop}
    \label{prop:random_matrices}
    We have,
    \begin{align}
        \begin{split}
            \tightsinglesubstack{\expect}{X \sim \orbit_{\ex}}\, \Phi_\yx(X)
            &= \frac{\vp_\yx(x)}{\dim V_\yx}  \id_{V_\yx}, \\
            \tightdoublesubstack{\expect}{X \sim \orbit_{\ex}}{Y \sim \orbit_{\ey}}\,  \frac{\vp^{\ya \yb}_\yc(X, Y)}{\dim V_\yc}
            &= c_{\ya \yb}^\yc \frac{\vp_\ya(x)}{\dim V_\ya}\frac{\vp_\yb(y)}{\dim V_\yb}, \\
            \tighttriplesubstack{\expect}{X \sim \orbit_{\ex}}{Y \sim \orbit_{\ey}}{Z \sim \orbit_{\ez}}\, \frac{\vp^{\ya \yb \yd}_{\yc \ye \yf}(X, Y, Z)}{\dim V_\ye}
            &= \norm{\abcdefsymbol}_2^2 \frac{\vp_\ya(x)}{\dim V_\ya}\frac{\vp_\yb(y)}{\dim V_\yb}
            \frac{\vp_\yd(y)}{\dim V_\yd}.
        \end{split}
    \end{align}
\end{prop}
\begin{proof}
    For the first equality, by Schur's Lemma we have
    \begin{equation}
        \tightsinglesubstack{\expect}{X \sim \orbit_{\ex}}\, \Phi_\yx(X) = \int_{\orbit_x} \Phi_\yx(X) \diff\mu_x = \vp_\yx(x) \frac{\id_{V_\yx}}{\dim V_\yx},
    \end{equation}
    where the proportionality is fixed by taking traces and $\vp_\yx(x)  = \Tr[\Phi_\yx(X)]$ by definition.
    For the second equality, we compute
    \begin{align}
        \begin{split}
            \tightdoublesubstack{\expect}{X \sim \orbit_{\ex}}{Y \sim \orbit_{\ey}}\, \frac{\vp^{\ya \yb}_\yc(X, Y)}{\dim V_\yc}
            &= \frac{1}{\dim V_\yc} \tightdoublesubstack{\expect}{X \sim \orbit_{\ex}}{Y \sim \orbit_{\ey}}\, \Tr (\Phi_\ya(X) \ot \Phi_\yb(Y)) \Pi_{\ya \yb}^\yc, \\
            &= \frac{\vp_\ya(x)}{\dim V_\ya}\frac{\vp_\yb(y)}{\dim V_\yb} \frac{1}{\dim V_\yc} \Tr \Pi_{\ya \yb}^\yc, \\
            &= c_{\ya \yb}^\yc \frac{\vp_\ya(x)}{\dim V_\ya}\frac{\vp_\yb(y)}{\dim V_\yb}.
        \end{split}
    \end{align}
    Finally, for the third equality we compute
    \begin{align}
        \begin{split}
            &\tighttriplesubstack{\expect}{X \sim \orbit_{\ex}}{Y \sim \orbit_{\ey}}{Z \sim \orbit_{\ez}}\, \frac{\vp^{\ya \yb \yd}_{\yc \ye \yf}(X, Y, Z)}{\dim V_\ye}\\
            &\qquad=
            \frac{1}{{\dim V_\ye}} \tighttriplesubstack{\expect}{X \sim \orbit_{\ex}}{Y \sim \orbit_{\ey}}{Z \sim \orbit_{\ez}}\,
            \Tr \projr (\Phi_\ya(X) \ot \Phi_\yb(Y) \ot \Phi_\yd(Z))\projl, \\
            &\qquad=
            \frac{1}{{\dim V_\ye}}  \frac{\vp_\ya(x)}{\dim V_\ya}\frac{\vp_\yb(y)}{\dim V_\yb}
            \frac{\vp_\yd(y)}{\dim V_\yd} \Tr \projr \projl, \\
            &\qquad=
            \norm{\abcdefsymbol}_2^2 \frac{\vp_\ya(x)}{\dim V_\ya}\frac{\vp_\yb(y)}{\dim V_\yb}
            \frac{\vp_\yd(y)}{\dim V_\yd}.
        \end{split}
    \end{align}
\end{proof}
Note that even though functions $\vp^{\ya \yb \yd}_{\yc \ye \yf}(X, Y, Z)$ are complex values, their expectation values are real and positive.
We will not use \cref{prop:random_matrices} in subsequent proofs directly; rather we will use the following consequence of \cref{prop:random_matrices} which enables us to place lower bounds on $\abs{\vp^{\ya\yb\yd}_{\yc\ye\yf}(X,Y,Z)}$ which complement the upper bounds derived earlier.
\begin{lem}
    \label{lem:estimate_E_x_y_z_bound}
    For all $(\ya,\yb,\yc,\yd,\ye,\yf) \in (\undomweightsp{n})^{\times 6}$ and all $(x,y,z) \in \mathbb R_n^{+}$,
    \begin{equation}
        \norm{\abcdefsymbol}_{2}^{2} \vp_{\ya}(\ex)\vp_{\yb}(\ey)\vp_{\yd}(\ez)
        \leq \polyk{3\binom{n}{2}}
        \tighttriplesubstack{\mathrm{max}}{X \in \orbit_{\ex}}{Y \in \orbit_{\ey}}{Z \in \orbit_{\ez}}
        \abs{\vp^{\ya\yb\yd}_{\yc\ye\yf}(X,Y,Z)}.
    \end{equation}
\end{lem}
\begin{proof}
    By \cref{prop:random_matrices}, we have
    \begin{align}
        \norm{\abcdefsymbol}_{2}^{2} \vp_{\ya}(\ex)\vp_{\yb}(\ey)\vp_{\yd}(\ez)
        &=
        \frac{d_\ya d_\yb d_\yd}{d_\ye}
        \tighttriplesubstack{\expect}{X \sim \orbit_{\ex}}{Y \sim \orbit_{\ey}}{Z \sim \orbit_{\ez}}\,
        \vp^{\ya\yb\yd}_{\yc\ye\yf}(X,Y,Z),
    \end{align}
    where $d_{\yx} \coloneqq \dim V_{\yx}$.
    The claim then follows from the bound
    \begin{align}
        \tighttriplesubstack{\expect}{X \sim \orbit_{\ex}}{Y \sim \orbit_{\ey}}{Z \sim \orbit_{\ez}}\,
        \vp^{\ya\yb\yd}_{\yc\ye\yf}(X,Y,Z)
        &\leq
        \tighttriplesubstack{\mathrm{max}}{X \in \orbit_{\ex}}{Y \in \orbit_{\ey}}{Z \in \orbit_{\ez}}
        \abs{\vp^{\ya\yb\yd}_{\yc\ye\yf}(X,Y,Z)},
    \end{align}
    along with the dimension estimates $\dim (V_\yx) \leq (\syx + 1)^{\binom{n}{2}} \leq \polyk{\binom{n}{2}}$ for $\yx \in \{\ya,\yb,\yd\}$ from \cref{eq:unitary_irrep_dim_bound} and $\dim (V_{\ye}) \geq 1$.
\end{proof}

\section{Inequalities for tetrahedral eigenvalues}
\label{sec:proof_tet_iff}

Here we prove the main result of this paper, \cref{thm:tet_iff}.
For convenience, we recall \cref{thm:tet_iff} from the introduction.
\tetiffthm*

\begin{proof}
    Our proof of \cref{thm:tet_iff} is divided into two parts.
    First we prove the necessity of \cref{eq:tet_iff} in \cref{sec:tet_necess}.
    Then afterwards we prove the distance bound in \cref{eq:tet_dist_bound}.
    Providing proofs for these two claims yields the whole theorem because the satisfaction of \cref{eq:tet_iff} for all $(\ya,\yb,\yd)$ with $\sya + \syb + \syd = k$ implies \cref{eq:tet_dist_bound} holds for all $k$, and thus the $k \to \infty$ limit implies $D\!\tightparensixj{\ea & \eb & \ec \\ \ed & \ee & \ef}=0$ (because $\lim_{k \to \infty} \ln \polyk{}/k =0$) which means $(\ea, \eb, \ec, \ed, \ee, \ef) \in \Tetra(n)$.
\end{proof}

\subsection{Necessary conditions}
\label{sec:tet_necess}

First, we assume that there exists matrices $(A,B,C,D,E,F)$ with eigenvalues
$(\ea,\eb,\ec,\ed,\ee,\ef)$ satisfying the tetrahedral constraints $A+B=C$, $B+D=F$, $E=A+F$, and $E=C+D$.
The result of \cref{prop:varphi_tet_main_ineq} applied to $X \mapsto A, Y \mapsto B$, and $Z \mapsto D$ produces
\begin{align}
    \begin{split}
        \left|\vp^{\ya\yb\yd}_{\yc\ye\yf}(A,B,D)\right|
        &\leq \norm{\abcdefsymbol}_{\infty} \vp^{\ya\yb\yd}_{\_\ye\yf}(A,B,D)^{\frac{1}{2}} \vp^{\ya\yb\yd}_{\yc\ye\_}(A,B,D)^{\frac{1}{2}}.
    \end{split}
\end{align}
Then four applications of \cref{prop:tet_psum_ineq} produces
\begin{align}
    \begin{split}
        \vp^{\ya\yb\yd}_{\_\ye\yf}(A, B, D) &\leq \left(\vp_\ya(A) \vp_\yf(F)\right)^{\frac{2}{3}} \vp_\ye(E)^{\frac{1}{3}}, \\
        \vp^{\ya\yb\yd}_{\yc\ye\_}(A,B,D) &\leq \left( \vp_\yc(C) \vp_\yd(D)\right)^{\frac{2}{3}} \vp_\ye(E)^{\frac{1}{3}}.
    \end{split}
\end{align}
These bounds combined yields:
\begin{equation}
    \label{eq:vp_ABD_abs_bound}
    \left|\vp^{\ya\yb\yd}_{\yc\ye\yf}(A,B,D)\right|
    \leq \norm{\abcdefsymbol}_{\infty}
    \left( \vp_{\ya}(A) \vp_{\yd}(D) \vp_{\yf}(F)\vp_{\yc}(C)\vp_{\ye}(E)\right)^{\frac{1}{3}}.
\end{equation}
We then compute,
\begin{align}
    \begin{split}
        \vp_{\ya}(\ea)\vp_{\yb}(\eb)\vp_{\yd}(\ed)
        &= \sum_{\yc,\ye,\yf}\vp^{\ya\yb\yd}_{\yc\ye\yf}(A,B,D), \\
        &\leq \sum_{\yc,\ye,\yf} \left|\vp^{\ya\yb\yd}_{\yc\ye\yf}(A,B,D)\right|, \\
        &= \sum_{\yc,\ye,\yf} \norm{\abcdefsymbol}_{\infty}
        \left( \vp_{\ya}(\ea) \vp_{\yd}(\ed) \vp_{\yf}(\ef)\vp_{\yc}(\ec)\vp_{\ye}(\ee)\right)^{\frac{1}{3}},
    \end{split}
\end{align}
as desired.

\subsection{Distance bound \& sufficiency}

Let $(a,b,c,d,e,f) \in (\eigsp{n})^6$ which satisfy the trace conditions and also \cref{eq:tet_iff}:
\begin{equation}
    \label{eq:abd_multiplied}
    \vp_{\ya}(\ea)\vp_{\yb}(\eb)\vp_{\yd}(\ed)
    \leq \sum_{\yc,\ye,\yf} \norm{\abcdefsymbol}_{\infty}
    \left(
        \vp_{\ya}(\ea)\vp_{\yf}(\ef)
        \vp_{\yc}(\ec)\vp_{\yd}(\ed)
        \vp_{\ye}(\ee)
    \right)^{\frac{1}{3}}
\end{equation}
for all $(\ya, \yb, \yd) \in (\undomweightsp{n})^3$.
We will consider triples with $\sya + \syb + \syc = k$ for some fixed integer $k$.
Denote by $(\yc^*, \ye^*, \yf^*)$ the triple for which the summand on the right hand side is maximal.
We have,
\begin{equation}
    \label{eq:first_abd_bound}
    \vp_{\ya}(\ea)\vp_{\yb}(\eb)\vp_{\yd}(\ed)
    \leq \polyk{3n} \norm{\abcdefsymbolcefstarred}_{\infty}
    \left(
        \vp_{\ya}(\ea)\vp_{\yf^*}(\ef)
        \vp_{\yc^*}(\ec)\vp_{\yd}(\ed)
        \vp_{\ye^*}(\ee)
    \right)^{\frac{1}{3}}
    ,
\end{equation}
where $\polyk{3n}$ serves as an upper bound on the number of terms in the sum in \cref{eq:abd_multiplied}.
\Cref{eq:first_abd_bound} then implies strict positivity
\begin{equation}
    \vp_{\ya}(\ea)\vp_{\yb}(\eb)\vp_{\yd}(\ed) > 0 \implies \norm{\abcdefsymbolcefstarred}_{\infty} >0,
\end{equation}
which will enable us to divide by $\norm{\abcdefsymbolcefstarred}$ momentarily.

At the same time, by \cref{lem:estimate_E_x_y_z_bound}, we have
\begin{align}
    \norm{\abcdefsymbolcefstarred}_{2}^{2} \vp_{\ya}(\ea)\vp_{\yb}(\eb)\vp_{\yd}(\ed)
    &\leq
    \polyk{3\binom{n}{2}}
    \abs{\vp^{\ya\yb\yd}_{\yc^*\ye^*\yf^*}(A^*,B^*,D^*)}
\end{align}
where $(A^*, B^*, D^*) \in \orbit_{a} \times \orbit_{b} \times \orbit_{d}$ is a triple of matrices which realizes a maximum of $\abs{\vp^{\ya\yb\yd}_{\yc^*\ye^*\yf^*}(A, B, D)}$.
Then, applying \cref{eq:vp_ABD_abs_bound} here yields
\begin{align}\label{eq:second_abd_bound}
    \begin{split}
        &\norm{\abcdefsymbolcefstarred}_{2}^{2} \vp_{\ya}(\ea)\vp_{\yb}(\eb)\vp_{\yd}(\ed)\\
        &\quad\leq
        \polyk{3\binom{n}{2}}
        \norm{\abcdefsymbolcefstarred}_{\infty}
        \left(\vp_{\ya}(A^*)\vp_{\yf^*}(F^*)\vp_{\yc^*}(C^*)\vp_{\yd}(D^{*})\vp_{\ye^*}(E^*)\right)^{\frac{1}{3}}.
    \end{split}
\end{align}
where $C^*\coloneqq A^*+B^*, E^*\coloneqq A^*+B^*+D^*, F^*\coloneqq B^*+D^*$.

Now we (i) multiply the estimate in \cref{eq:first_abd_bound} by the estimate in \cref{eq:second_abd_bound}, (ii) divide by $\norm{\abcdefsymbol}_{2}^{2}$ and use the standard operator norm estimate $\norm{\abcdefsymbol}_{2}^{2} \leq \norm{\abcdefsymbol}_{\infty}^{2}$, and (iii) take a square root to obtain:
\begin{align}\label{eq:varphi_sixth_root}
    \begin{split}
        \vp_{\ya}(\ea)\vp_{\yb}(\eb)\vp_{\yd}(\ed)
        &\leq \polyk{\frac{3}{4}n(n+1)}
        \left(
        \vp_{\ya}(\ea)\vp_{\yf^*}(\ef)
        \vp_{\yc^*}(\ec)\vp_{\yd}(\ed)
        \vp_{\ye^*}(\ee)
        \right)^{\frac{1}{6}} \\
        &\qquad\times
        \left(\vp_{\ya}(A^*)\vp_{\yf^*}(F^*)\vp_{\yc^*}(C^*)\vp_{\yd}(D^*)\vp_{\ye^*}(E^*)\right)^{\frac{1}{6}}.
    \end{split}
\end{align}
From here, we can apply \cref{prop:eig_sep}, or rather, the estimate from \cref{eq:weak_eig_sep_varphi} to relax the above estimate even further.
For instance, we have by \cref{eq:weak_eig_sep_varphi} and $\syf \leq k$ the estimate
\begin{align}
    \begin{split}
        \left(\vp_{\yf^*}(\ef) \vp_{\yf^*}(F^*)\right)^{\frac{1}{2}}
        &\leq (\syf+1)^{\binom{n}{2}} \frac{1}{\syf!} \BC(\ef, \eig(F^*))^{\syf}, \\
        &\leq \polyk{\binom{n}{2}} \frac{1}{\syf!} \BC(\ef, \eig(F^*))^{\syf}.
    \end{split}
\end{align}
Also, as $\eig(A^*)=a$, we have $\vp_{\ya}(A^*)=\vp_\ya(\ea)$, and thus by \cref{eq:varphi_normalization}
\begin{equation}
    \left(\vp_\ya(\ea) \vp_{\ya}(A^*)\right)^{\frac{1}{2}} =
    \vp_\ya(\ea) \leq \frac{\tea^{\sya}}{\sya!}.
\end{equation}
Putting these two estimates above together, we arrive at
\begin{align}
    \begin{split}
        &\left( \vp_\ya(\ea) \vp_{\ya}(A^*) \vp_{\yf^*}(\ef) \vp_{\yf^*}(F^*)\right)^{\frac{1}{2}} \\
        &\qquad\leq \frac{1}{\sya!\syf!} \polyk{\binom{n}{2}} \tea^{\sya} \BC(\ef, \eig(F^*))^{\syf}, \\
        &\qquad\leq \frac{1}{k!} \polyk{\binom{n}{2}} (\tea + \BC(\ef, \eig(F^*)))^{k}, \\
        &\qquad= \frac{1}{k!} \polyk{\binom{n}{2}} \BC(\ea \oplus \ef, \ea \oplus \eig(F^*))^k, \\
        &\qquad\leq \frac{\tee^k}{k!} \polyk{\binom{n}{2}} \exp\left( - \frac{k}{8\tee^2} \norm{\ef- \eig(F^*)}_1^2\right).
    \end{split}
\end{align}
Similarly, we obtain
\begin{align}
    \begin{split}
        &\left( \vp_\ya(\ed) \vp_{\ya}(D^*) \vp_{\yc^*}(\ec) \vp_{\yc^*}(C^*)\right)^{\frac{1}{2}}\\
        &\qquad\leq
        \frac{\tee^k}{k!} \polyk{\binom{n}{2}} \exp\left(-\frac{k}{8\tee^2} \norm{\ec - \eig(C^*)}_1^2\right),
    \end{split}
\end{align}
and
\begin{equation}
    \left( \vp_\ye(\ee) \vp_{\ye}(E^*) \right)^{\frac{1}{2}} \leq
    \frac{\tee^k}{k!} \polyk{\binom{n}{2}}
    \exp\left( - \frac{k}{8\tee^2} \norm{\ee- \eig(E^*)}_1^2\right).
\end{equation}
Applying these bounds to the inequality in \cref{eq:varphi_sixth_root} implies
\begin{equation}
    \vp_{\ya}(\ea) \vp_{\yb}(\eb) \vp_{\yd}(\ed) \leq
    \frac{\tee^k}{k!} \polyk{\frac{1}{4}n(5n+1)}
    \exp\left( - \frac{k}{24\tee^2} D\!\tightparensixj{\ea & \eb & \ec \\ \ed & \ee & \ef}^2 \right).
\end{equation}
Here we have used the fact that the distance measure satisfies
\begin{equation}
    D\!\tightparensixj{\ea & \eb & \ec \\ \ed & \ee & \ef}^{2}
    \leq
    \norm{\ec- \eig(C^*)}_1^2 + \norm{\ee- \eig(E^*)}_1^2 + \norm{\ef- \eig(F^*)}_1^2.
\end{equation}
Summing over the triples $(\ya, \yb, \yd)$ with $\sya+\syb+\syd=k$ yields
\begin{equation}
    1 \leq
    \polyk{\frac{1}{4}n(5n+13)}
    \exp\left( - \frac{k}{24\tee^2} D\!\tightparensixj{\ea & \eb & \ec \\ \ed & \ee & \ef}^2 \right),
\end{equation}
where we have estimated by $\polyk{3n}$ the number of summands.
Altogether, this implies%
\begin{equation}
    \label{eq:final_estimate}
    \frac{1}{\tee} D\!\tightparensixj{\ea & \eb & \ec \\ \ed & \ee & \ef} \leq 6 \sqrt{3} n \sqrt{\frac{\ln \polyk{} }{k}},
\end{equation}
because $\sqrt{24(\frac{1}{4}n(5n+13))} \leq 6 \sqrt{3} n$ for all $n \geq 1$.

\section{Asymptotics of \texorpdfstring{$U(n)$ $6j$}{U(n) 6j} symbols}
\label{sec:proof_6j_estimates}

In this section we provide a proof of \cref{thm:6j_estimates},
specifically, \cref{sec:6j_inverse_poly} provides the proof of the inverse polynomial lower bound while \cref{sec:6j_exp} provides the proof of the exponential upper bound.
Before proceeding, we prove the following lemma which will be used in both portions of the proof.
\begin{lem}
    \label{lem:asym_sixj_helper}
    Let $X,Y,Z \in \hermp{n}$ and let $(\ea,\eb,\ec,\ed,\ee,\ef)$ be the eigenvalues of the matrices $(X, Y, X+Y, Z, X+Y+Z,Y+Z)$.
    Then for all $(\ya,\yb,\yc,\yd,\ye,\yf) \in (\undomweightsp{n})^{\times 6}$ satisfying the size conditions of \cref{eq:size_conditions_tet} with $\abs{\ye} = k$, we have
    \begin{align}
        \label{eq:hex_rel_ent_bound}
        \frac{k!}{\tee^k}\abs{\vp^{\ya\yb\yd}_{\yc\ye\yf}(X,Y,Z)}
        &\leq \polyk{n(n-1)} \exp\left(- k R_{\ea\eb\ec\ed\ee\ef}^{\ya\yb\yc\yd\ye\yf}\right),
    \end{align}
    where
    \begin{align}
        \begin{split}
            R_{\ea\eb\ec\ed\ee\ef}^{\ya\yb\yc\yd\ye\yf}
            = \frac{1}{4}\bigg[
            \kl{\frac{\ya \oplus \yb \oplus \yd}{k}}{\frac{\ea\oplus \eb \oplus \ed}{\tee}}
            &+ \kl{\frac{\yc \oplus \yd}{k}}{\frac{\ec \oplus \ed}{\tee}}+ \\
            \cdots + \kl{\frac{\ya \oplus \yf}{k}}{\frac{\ea\oplus \ef}{\tee}}
            &+ \kl{\frac{\ye}{k}}{\frac{\ee}{\tee}}
            \bigg].
        \end{split}
    \end{align}
    Note that $R_{\ea\eb\ec\ed\ee\ef}^{\ya\yb\yc\yd\ye\yf} \geq 0$ with equality holding only if
    \begin{equation}
        \frac{1}{k} (\ya, \yb, \yc, \yd, \ye, \yf)
        =
        \frac{1}{\tee} (\ea, \eb, \ec, \ed, \ee, \ef).
    \end{equation}
\end{lem}
\begin{proof}
    Using \cref{prop:varphi_tet_main_ineq}, \cref{prop:tet_psum_ineq} and $\norm{\abcdefsymbol}_{\infty} \leq 1$, we obtain
    \begin{equation}
        \label{eq:abs_quad_bound}
        \abs{\vp^{\ya\yb\yd}_{\yc\ye\yf}(X,Y,Z)}
        \leq
        \left(\vp_{\ya}(\ea)\vp_{\yb}(\eb)\vp_{\yd}(\ed) \times \vp_{\yc}(\ec)\vp_{\yd}(\ed) \times \vp_{\ye}(\ee) \times \vp_{\ya}(\ea)\vp_{\yf}(\ef)\right)^{\frac{1}{4}}.
    \end{equation}
    Using the multinomial distribution bound from \cref{prop:multinomial_bound}, we observe
    \begin{align}
        \begin{split}
            \vp_{\ya}(\ea)\vp_{\yb}(\eb)\vp_{\yd}(\ed)
            &\leq \polyk{3\binom{n}{2}}
            \frac{\ea_1^{\ya_1}}{\ya_1!}\cdots\frac{\ea_n^{\ya_n}}{\ya_n!}
            \frac{\eb_1^{\yb_1}}{\yb_1!}\cdots\frac{\eb_n^{\yb_n}}{\yb_n!}
            \frac{\ed_1^{\yd_1}}{\yd_1!}\cdots\frac{\ed_n^{\yd_n}}{\yd_n!}, \\
            &= \polyk{3\binom{n}{2}} \frac{\tee^{k}}{k!}
            m_{\ya \oplus \yb \oplus \yd}(\ea \oplus \eb \oplus \ed).
        \end{split}
    \end{align}
    Then applying the relative entropy bound from the proof of \cref{prop:eig_est} yields
    \begin{align}
        \frac{k!}{\tee^k}{\vp_{\ya}(\ea)\vp_{\yb}(\eb)\vp_{\yd}(\ed)}
        &\leq \polyk{3\binom{n}{2}}\exp\left[ - k \kl{\frac{\ya \oplus \yb \oplus \yd}{k}}{\frac{\ea\oplus \eb \oplus \ed}{\tee}}\right].
    \end{align}
    Analogously, we obtain
    \begin{align}
        \frac{k!}{\tee^k}{\vp_{\yc}(\ec)\vp_{\yd}(\ed)}
        &\leq\polyk{2\binom{n}{2}} \exp\left[ - k \kl{\frac{\yc \oplus \yd}{k}}{\frac{\ec \oplus \ed}{\tee}}\right], \\
        \frac{k!}{\tee^k}{\vp_{\ya}(\ea)\vp_{\yf}(\ef)}
        &\leq\polyk{2\binom{n}{2}}\exp\left[ - k \kl{\frac{\ya \oplus \yf}{k}}{\frac{\ea\oplus \ef}{\tee}}\right], \\
        \frac{k!}{\tee^k}{\vp_{\ye}(\ee)}
        &\leq\polyk{\binom{n}{2}}\exp\left[ - k \kl{\frac{\ye}{k}}{\frac{\ee}{\tee}}\right].
    \end{align}
    Applying these bounds to \cref{eq:abs_quad_bound} we obtain \cref{eq:hex_rel_ent_bound}.
\end{proof}

\subsection{Inverse polynomial lower bound}
\label{sec:6j_inverse_poly}

To begin we note that for all $(X,Y,Z) \in (\hermp{n})^{3}$ and $(\ya,\yb,\yc,\yd,\ye,\yf) \in (\undomweights{n})^{6}$, we have the following bound between $\vp^{\ya\yb\yd}_{\yc\ye\yf}(X,Y,Z)$ and $\abcdefsymbol$:
\begin{align}
    \label{eq:vp_tet_bound_sixj}
    \begin{split}
        \abs{\vp^{\ya\yb\yd}_{\yc\ye\yf}(X,Y,Z)}
        &\leq \norm{\abcdefsymbol}_{\infty} \vp^{\ya\yb\yd}_{\_\ye\yf}(X,Y,Z)^{\frac{1}{2}} \vp^{\ya\yb\yd}_{\yc\ye\_}(X,Y,Z)^{\frac{1}{2}}, \\
        &\leq \norm{\abcdefsymbol}_{\infty} \varphi_{\ye}(X+Y+Z), \\
        &\leq \norm{\abcdefsymbol}_{\infty} \frac{\Tr(X+Y+Z)^k}{k!},
    \end{split}
\end{align}
where the first inequality is \cref{prop:varphi_tet_main_ineq}, the second inequality is \cref{prop:tet_psum_ineq} and the third inequality follows from the normalization of $\varphi_{\ye}$ (\cref{eq:varphi_normalization}).

Next we assume $(\ea,\eb,\ec,\ed,\ee,\ef)$ are the eigenvalues of non-negative Hermitian matrices $(X,Y,X+Y,Z,X+Y+Z,Y+Z)$.
Summing $\vp^{\ya\yb\yd}_{\yc\ye\yf}(X,Y,Z)$ over all tuples of Young diagrams satisfying the size conditions in \cref{eq:size_conditions_tet} produces the normalization condition
\begin{equation}
    \label{eq:size_conditions_tet_normalization}
    \sum_{\substack{\ya,\yb,\yd \\ \yc,\ye,\yf \\ \sye = k}} \vp^{\ya\yb\yd}_{\yc\ye\yf}(X,Y,Z) = \frac{\tee^k}{k!}.
\end{equation}
Now for each positive integer $k$, let $(\ya_k,\yb_k,\yc_k,\yd_k,\ye_k,\yf_k) \in (\undomweightsp{n})^{6}$ denote the tuple which maximizes the absolute value of the summand $\vp^{\ya\yb\yd}_{\yc\ye\yf}(X,Y,Z)$ in \cref{eq:size_conditions_tet_normalization} above.
Using this maximum, we bound each summand to conclude
\begin{equation}
    \label{eq:inverse_poly_varphi_lower}
    \frac{\tee^k}{k!}
    = \sum_{\substack{\ya,\yb,\yd \\ \yc,\ye,\yf \\ \sye = k}} \vp^{\ya\yb\yd}_{\yc\ye\yf}(X,Y,Z)
    \leq \polyk{6\rank[\ee]} \abs{\vp^{\ya_k\yb_k\yd_k}_{\yc_k\ye_k\yf_k}(X,Y,Z)},
\end{equation}
where $\polyk{6\rank[\ee]}$ is an upper bound the total number of terms in the summation for which $\vp^{\ya\yb\yd}_{\yc\ye\yf}(X,Y,Z) \neq 0$.

Third we combine the lower bound in \cref{eq:inverse_poly_varphi_lower} with the upper bound in \cref{eq:vp_tet_bound_sixj} (noting that $\tee = \Tr(X+Y+Z)$) to obtain
\begin{equation}
    1 \leq \polyk{6\rank[\ee]} \norm{\tightcurlysixj{ \ya_k & \yb_k & \yc_k \\ \yd_k &  \ye_k &  \yf_k}}_{\infty},
\end{equation}
as stated in \cref{eq:n_independent_inverse_poly}.
Finally, to prove the convergence portion of the claim, we combine \cref{lem:asym_sixj_helper} with \cref{eq:inverse_poly_varphi_lower} to obtain
\begin{equation}
    \exp\left(- k R_{\ea\eb\ec\ed\ee\ef}^{\ya_k\yb_k\yc_k\yd_k\ye_k\yf_k}\right) \leq \frac{1}{\polyk{6\rank[\ee] + n(n-1)}},
\end{equation}
which implies $\lim_{k\to \infty} R_{\ea\eb\ec\ed\ee\ef}^{\ya_k\yb_k\yc_k\yd_k\ye_k\yf_k} = 0$ which implies the convergence claim of \cref{thm:6j_estimates}.

\subsection{Exponential upper bound}
\label{sec:6j_exp}

We now prove the exponential decay on $\norm{\abcdefsymbolk}_{\infty}$ whenever $(\ya_k,\yb_k,\yc_k,\yd_k,\ye_k,\yf_k)$ converges to eigenvalues which are not the eigenvalues of any tetrahedra of Hermitian matrices.
Without loss of generality, we can assume the Young diagrams $(\ya_k,\yb_k,\yc_k,\yd_k,\ye_k,\yf_k)$ satisfy the size constraints \cref{eq:size_conditions_tet} as before (since otherwise $\abcdefsymbolk = 0$).
To begin, \cref{lem:estimate_E_x_y_z_bound} gives
\begin{equation}
    \norm{\abcdefsymbolk}_{2}^{2} \vp_{\ya_k}(\ea)\vp_{\yb_k}(\eb)\vp_{\yd_k}(\ed)
    \leq
    \polyk{3\binom{n}{2}}
    \tighttriplesubstack{\mathrm{max}}{X \in \orbit_{\ea}}{Y \in \orbit_{\eb}}{Z \in \orbit_{\ed}}
    \abs{\vp^{\ya_k\yb_k\yd_k}_{\yc_k\ye_k\yf_k}(X,Y,Z)}.
\end{equation}
Now we freely choose, for each $k$, a particular triple of eigenvalues $(\ea^*,\eb^*,\ed^*)$ which ensures
\begin{equation}
    \frac{k!}{\Tr[\ee^{*}]^k} \vp_{\ya_k}(\ea^*)\vp_{\yb_k}(\eb^*)\vp_{\yd_k}(\ed^*) \geq \frac{1}{\polyk{3n}}.
\end{equation}
where we are denoting the sum $\Tr[\ea^{*}]+\Tr[\eb^*]+\Tr[\ed^*]$ with $\Tr[\ee^{*}]$ by a slight abuse of notation as no eigenvalues $\ee^*$ have been chosen here.
That such a choice can always be made follows from the normalization condition
\begin{equation}
    \sum_{\substack{\ya,\yb,\yd \\ \sya+\syb+\syd = k}}\vp_{\ya}(\ea)\vp_{\yb}(\eb)\vp_{\yd}(\ed) = \frac{(\tea+\teb+\ted)^k}{k!},
\end{equation}
and noting that there are at most $\polyk{3n}$ such triples $(\ya,\yb,\yd)$ of Young diagrams in the sum (indeed it suffices to choose $(\ea^*,\eb^*,\ed^*)$ to maximize $\vp_{\ya_k}(\ea)\vp_{\yb_k}(\eb)\vp_{\yd_k}(\ed)$).
Therefore, we obtain the upper bound
\begin{equation}
    \norm{\abcdefsymbolk}_{2}^{2}
    \leq
    \polyk{3\binom{n}{2}+3n} \frac{\Tr[\ee^{*}]^k}{k!}
    \tighttriplesubstack{\mathrm{max}}{X \in \orbit_{\ea^{*}}}{Y \in \orbit_{\eb^{*}}}{Z \in \orbit_{\ed^{*}}}
    \abs{\vp^{\ya_k\yb_k\yd_k}_{\yc_k\ye_k\yf_k}(X,Y,Z)}.
\end{equation}
Then applying \cref{lem:asym_sixj_helper} produces
\begin{equation}
    \norm{\abcdefsymbolk}_{2}^{2}
    \leq
    \polyk{3\binom{n}{2}+3n + n(n-1)}
    \exp\left( - k
        \tighttriplesubstack{\mathrm{min}}{X \in \orbit_{\ea^{*}}}{Y \in \orbit_{\eb^{*}}}{Z \in \orbit_{\ed^{*}}}
        R_{\ea^{*}\eb^{*}\ec^{*}\ed^{*}\ee^{*}\ef^{*}}^{\ya_k\yb_k\yc_k\yd_k\ye_k\yf_k}
    \right).
\end{equation}
where $\ec^{*} = \exp(X+Y), \ef^{*} = \eig(Y+Z)$, and $\ee^{*} = \eig(X+Y+Z)$.

Finally, since it was assumed that $\frac{1}{k}(\ya_k,\yb_k,\yc_k,\yd_k,\ye_k,\yf_k)$ are converging to a point outside $\Tetra^{+}(n)$, there must exist a sufficiently large $K$ such that for all $k \geq K$ and all $(\ea',\eb',\ec',\ed',\ee',\ef') \in \Tetra^{+}(n)$, we have a separation $R_{\ea'\eb'\ec'\ed'\ee'\ef'}^{\ya_k\yb_k\yc_k\yd_k\ye_k\yf_k} \geq r$ for some positive constant $r > 0$.
Since $(\ea^{*},\eb^{*},\ec^{*},\ed^{*},\ee^{*},\ef^{*}) \in \Tetra(n)$, we can conclude for such sufficiently large $k \geq K$, $R_{\ea^{*}\eb^{*}\ec^{*}\ed^{*}\ee^{*}\ef^{*}}^{\ya_k\yb_k\yc_k\yd_k\ye_k\yf_k} \geq r$ and therefore
\begin{equation}
    \norm{\abcdefsymbolk}_{2}^{2} \leq \polyk{3\binom{n}{2}+3n + n(n-1)}  \exp\left( - k r\right),
\end{equation}
which proves the desired exponential decay on the operator norm because of the standard norm estimate $\norm{\abcdefsymbolk}_{\infty} \leq \norm{\abcdefsymbolk}_{2}$.

\appendix

\section{Symmetries of tetrahedral eigenvalues}
\label{sec:symmetries}

There are numerous types of symmetries present in the statement of the tetrahedral Horn problem.
The purpose of this section is to describe the scaling, inversion, tetrahedral, and shifting symmetries of the problem.

\begin{prop}[Scaling]
    \label{prop:scaling}
    Let $(\ea, \eb, \ec, \ed, \ee, \ef) \in \eigs{n}$ and let $s \in (0, \infty)$ be a positive real number.
    Then $(s\ea, s\eb, s\ec, s\ed, s\ee, s\ef) \in \Tetra(n)$ if and only $(\ea, \eb, \ec, \ed, \ee, \ef) \in \Tetra(n)$.
\end{prop}
\begin{proof}
    The statement follows from the fact that $\eig(s X) = s \eig(X)$ holds for all Hermitian matrices $X$ and positive scalars $s$ and that each of the constraints in \cref{eq:tet_cond} are homogeneous, e.g., we have $sA+sB-sC=0$ if and only if $A+B-C=0$ because $s \neq 0$.
\end{proof}

One may wonder whether it is possible to extend the scaling symmetry to negative values for $s$.
Indeed, this can be accomplished, provided one accounts for the fact that negating a Hermitian matrix both negates and reverses the sorting order of its eigenvalues.
\begin{prop}[Inversion]
    \label{prop:inversion}
    Let $(\ea, \eb, \ec, \ed, \ee, \ef)$ be a tuple of candidate eigenvalues.
    Given any list of non-increasing eigenvalues $\ex = (\ex_1 \geq \cdots \geq \ex_n)$, let $\ex^{*}$ denote the inverted list of resorted and negated eigenvalues: $\ex^{*} = (-\ex_n \geq \cdots \geq - \ex_1)$.
    Then $(\ea, \eb, \ec, \ed, \ee, \ef) \in \Tetra(n)$ if and only if $(\ea^{*}, \eb^{*}, \ec^{*}, \ed^{*}, \ee^{*}, \ef^{*}) \in \Tetra(n)$.
\end{prop}
\begin{proof}
    The statement holds because $\eig(-X) = \ex^{*}$ whenever $\eig(X) = \ex$ and moreover the tuple of matrices $(A,B,C,D,E,F)$ satisfies \cref{eq:tet_cond} if and only if the tuple of negated matrices $(-A,-B,-C,-D,-E,-F)$ satisfies \cref{eq:tet_cond}.
\end{proof}

The inversion symmetry \cref{prop:inversion} is just one element belonging to a group of $48$ symmetries consisting of inversions of eigenvalues, $\ex \rightarrow \ex^{*}$, or permutations of elements in the tuple $(\ea,\eb,\ec,\ed,\ee,\ef)$.

\begin{prop}[Signed permutation]
    \label{prop:tet_symmetry}
    The following transformations preserve $\Tetra(n)$:
    \begin{align}
        \label{eq:gens_of_tet_sym}
        \begin{split}
            T_1 \coloneqq (\ea, \eb, \ec, \ed, \ee, \ef)
            &\mapsto
            (\ea, \ef, \ee, \ed^{*}, \ec, \eb), \\
            T_2 \coloneqq (\ea, \eb, \ec, \ed, \ee, \ef)
            &\mapsto
            (\ec, \ee^*, \ed^{*}, \ef, \eb, \ea^*).
        \end{split}
    \end{align}
    Moreover, $T_1$ and $T_2$ generate a finite group of order $48$ with presentation
    \begin{equation}
        \langle T_1, T_2 \mid T_1^2, T_2^4, (T_1 T_2)^6, (T_2T_1T_2T_1T_2)^2 \rangle.
    \end{equation}
\end{prop}
\begin{proof}
    To begin, notice that both $T_1$ and $T_2$ are invertible (indeed the first operation has order two, while the second operation has order four).
    The claim that these two operations generate a finite group of order $48$ can be checked explicitly using computer software.
    Now suppose $(\ea,\eb,\ec,\ed,\ee,\ef)$ are the eigenvalues of a tuple of Hermitian matrices $(A,B,C,D,E,F)$ satisfying \cref{eq:tet_cond}.
    Then the tuple $(A',B',C',D',E',F') = (A,F,E,-D,C,B)$ (which has the eigenvalues $(\ea, \ef, \ee, \ed^{*}, \ec, \eb)$) also satisfies the tetrahedral conditions because
    \begin{align}
        \label{eq:signed_perm_rows}
        \begin{split}
            A'+B'-C' &= (A+F-E) = 0,  \\
            B'+D'-F' &= -(B+D-F) = 0, \\
            C'+D'-E' &= -(C+D-E) = 0, \\
            A'+F'-E' &= (A+B-C) = 0.
        \end{split}
    \end{align}
    Therefore, $(\ea,\eb,\ec,\ed,\ee,\ef) \in \Tetra(n)$ implies $(\ea^*, \ec, \eb, \ed, \ef, \ee) \not \in \Tetra(n)$ also.
    Moreover, as the operation $T_1$ mapping $(\ea,\eb,\ec,\ed,\ee,\ef)$ to $(\ea^*, \ec, \eb, \ed, \ef, \ee)$ is its own inverse, we conclude that $(\ea,\eb,\ec,\ed,\ee,\ef) \not \in \Tetra(n)$ implies $(\ea^*, \ec, \eb, \ed, \ef, \ee) \not \in \Tetra(n)$ also.
    A similar argument holds for $T_2$ which maps $(\ea,\eb,\ec,\ed,\ee,\ef)$ to $(\ec, \ee^*, \ed^{*}, \ef, \eb, \ea^*)$ where $(A',B',C',D',E',F')$ are given by the matrices $(C,-E,-D,F,B,-A)$.
\end{proof}
\begin{rem}
    The group $G$ of order $48$ appearing in \cref{prop:tet_symmetry} corresponds to the subgroup of all possible signed permutations of $6$ elements generated by permutations and negations of the six matrices $A,B,C,D,E$, and $F$ which preserves the set of linear constraints in \cref{eq:tet_cond}.
    The group $G$ is isomorphic to both the direct product $S_2 \times S_4$ and to the signed permutation group of three elements, i.e., $(S_2 \times S_2 \times S_2) \rtimes S_3$.

    The first isomorphism, $G \simeq S_2 \times S_4$, can be understood as $S_2$ acting by the inversion of all eigenvalues,
    \begin{equation}
        (\ea,\eb,\ec,\ed,\ee,\ef) \leftrightarrow (\ea^*,\eb^*,\ec^*,\ed^*,\ee^*,\ef^*),
    \end{equation}
    while $S_4$ acts by permutation of the \textit{vertices} of the tetrahedron in \cref{fig:tetrahedron_labelling_scheme}:
    each vertex permutation induces a permutation of the six directed edges with the possibility of inverting the orientation of a subset of the edges (such edge orientation reversals can be accounted for a corresponding inversion of eigenvalues).

    The second isomorphism, $G \simeq (S_2 \times S_2 \times S_2) \rtimes S_3$, can be understood by considering the triplet $(A,D), (B,E), (C,F)$ of non-adjacent pairs of edges in \cref{fig:tetrahedron_labelling_scheme}.
    Accordingly, it is helpful to arrange the tuple $(\ea,\eb,\ec,\ed,\ee,\ef)$ into a $2\times 3$ grid using the following identification:
    \begin{equation}
        \left(\tightsextuple{\ea & \eb & \ec & \ed & \ee & \ef}\right)
        \cong
        \left(\tightsextuple{\ea & \eb & \ec \\ \ed & \ee & \ef}\right).
    \end{equation}
    Then one can see that each $S_2$ factor in the direct product $S_2 \times S_2 \times S_2$ acts faithfully on the tuple by both negating a certain subset of the columns and permuting the upper and lower rows leaving just one column unchanged.
    \begin{equation}
        \left(\tightsextuple{\ea & \eb & \ec \\ \ed & \ee & \ef}\right)
        \leftrightarrow
        \left(\tightsextuple{\ea & \ee^* & \ef^* \\ \ed & \eb^* & \ec^*}\right),
        \qquad
        \left(\tightsextuple{\ea & \eb & \ec \\ \ed & \ee & \ef}\right)
        \leftrightarrow
        \left(\tightsextuple{\ed & \eb & \ef \\ \ea & \ee & \ec}\right),
        \qquad
        \left(\tightsextuple{\ea & \eb & \ec \\ \ed & \ee & \ef}\right)
        \leftrightarrow
        \left(\tightsextuple{\ed^* & \ee & \ec \\ \ea^* & \eb & \ef}\right).
    \end{equation}
    The symmetric group $S_3$ of order $3! = 6$ acts instead by permuting the columns and negating a subset of the entries in the top row, producing a subgroup of order $6$ of the form:
    \begin{equation}
        \begin{alignedat}{3}
            &\left(\tightsextuple{\ea & \eb & \ec \\ \ed & \ee & \ef}\right)
            \quad
            &&\left(\tightsextuple{\eb & \ec^* & \ea^* \\ \ee & \ef & \ed}\right)
            \quad
            &&\left(\tightsextuple{\ec^* & \ea & \eb^*  \\ \ef & \ed & \ee }\right)\\[0.5em]
            &\left(\tightsextuple{\ea^* & \ec & \eb \\ \ed & \ef & \ee}\right)
            \quad
            &&\left(\tightsextuple{\eb^* & \ea^*  & \ec^* \\ \ee & \ed & \ef}\right)
            \quad
            &&\left(\tightsextuple{\ec & \eb^* & \ea \\ \ef & \ee & \ed}\right)\\
        \end{alignedat}
    \end{equation}
    The symmetries in \cref{prop:tet_symmetry} also coincide with certain symmetries of the $6j$ operator $\abcdefsymbol$ under the involution $\yx \mapsto \yx^{*}$ corresponding taking the dual representation $V_{\yx} \mapsto (V_{\yx})^{*}$.
    In the case where $n = 2$, the involution map becomes the identity map (as every irreducible representation of $SU(2)$ is self-dual, or the reversal of a pair $(x,-x)$ of traceless eigenvalues is trivial), and thus the symmetry group collapses to a group of order $\abs{S_4} = 4! = 24$.
\end{rem}
\begin{prop}[Shifting symmetry]
    \label{prop:shifting}
    Let $(\ea, \eb, \ec, \ed, \ee, \ef) \in (\eigs{n})^{6}$ and let $x,y,z \in \mathbb R$ be fixed.
    Define shifted eigenvalues $(\tilde \ea, \tilde \eb, \tilde \ec, \tilde \ed, \tilde \ee,\tilde \ef)$ by
    \begin{equation}
        \label{eq:shifted_tet_eigenvalues}
        \begin{alignedat}{3}
             &\tilde \ea_i \coloneqq \ea_i + x, \qquad
            &&\tilde \eb_i \coloneqq \eb_i + y, \qquad
            &&\tilde \ec_i \coloneqq \ec_i + x + y, \\
             &\tilde \ed_i \coloneqq \ed_i + z, \qquad
            &&\tilde \ee_i \coloneqq \ee_i + x + y + z, \qquad
            &&\tilde \ef_i \coloneqq \ef_i + y + z.
        \end{alignedat}
    \end{equation}
    Then $(\tilde \ea, \tilde \eb, \tilde \ec, \tilde \ed, \tilde \ee,\tilde \ef) \in \Tetra(n)$ if and only if $(\ea, \eb, \ec, \ed, \ee, \ef) \in \Tetra(n)$.
\end{prop}
\begin{proof}
    Given a Hermitian matrix $A$ with eigenvalues $\eig(A) = \ea$ and a constant $w \in \mathbb R$, the eigenvalues of $A+wI$ (where $I$ is the identity matrix) are
    \begin{equation}
        \eig(A+wI) = (\ea_1 + w \geq \cdots \geq \ea_n + w).
    \end{equation}
    The claim then follows because the tuple of matrices
    \begin{equation}
        \begin{alignedat}{3}
             &\tilde A \coloneqq A + xI, \qquad
            &&\tilde B \coloneqq B + yI, \qquad
            &&\tilde C \coloneqq C + (x + y)I, \\
             &\tilde D \coloneqq D + zI, \qquad
            &&\tilde E \coloneqq E + (x + y + z)I, \qquad
            &&\tilde F \coloneqq F + (y + z)I,
        \end{alignedat}
    \end{equation}
    have eigenvalues $(\tilde \ea, \tilde \eb, \tilde \ec, \tilde \ed, \tilde \ee,\tilde \ef)$ and satisfy the tetrahedral constraints in \cref{eq:tet_cond} if and only if $(A,B,C,D,E,F)$ satisfy the tetrahedral constraints in \cref{eq:tet_cond}.
\end{proof}
As a corollary of the shifting symmetries (\cref{prop:shifting}) one concludes that to solve the tetrahedral Horn problem, in a certain sense, it suffices to solve the tetrahedral Horn problem merely for traceless matrices (\cref{cor:traceless}) or, alternatively, for non-negative or positive matrices (\cref{cor:non_negativity}).
\begin{cor}[Traceless eigenvalues]
    \label{cor:traceless}
    Let $(\ea, \eb, \ec, \ed, \ee, \ef)$ be a tuple of candidate eigenvalues and for any list of $n$ eigenvalues $\ex = (\ex_1 \geq \cdots \geq \ex_n)$, let $\ex^{0} \coloneqq \ex - {\tex}/n$ denote the traceless list of eigenvalues obtained by shifting $\ex$ by $-{\tex}/n$.
    Then $(\ea, \eb, \ec, \ed, \ee, \ef) \in \Tetra(n)$ if and only if (i) their traces $({\tea}, {\teb}, {\tec}, {\ted}, {\tee}, {\tef})$ are tetrahedral (they satisfy \cref{eq:tet_trace_cond}) and (ii) $(\ea^0, \eb^0, \ec^0, \ed^0, \ee^0, \ef^0) \in \Tetra(n)$.
\end{cor}
\begin{proof}
    This ``only if'' portion of the statement follows from both the tetrahedral trace conditions (\cref{eq:tet_trace_cond}) and applications of the shifting symmetry \cref{prop:shifting} with $(x,y,z)$ values
    \begin{equation}
        x = - n^{-1}\Tr(A), \qquad
        y = - n^{-1}\Tr(B) \qquad
        z = - n^{-1}\Tr(D).
    \end{equation}
    The ``if'' portion of the statement follows from the shifting symmetry applied to the negated $(x,y,z)$ values.
\end{proof}
The next corollary of the shifting symmetry proves that to solve the tetrahedral Horn problem it suffices to solve the tetrahedral Horn problem for non-negative (or even positive) eigenvalues only.
Indeed, one can take the shifting parameters $(x,y,z)$ appearing in \cref{prop:shifting} to be sufficiently large.
The following corollary establishes, in a certain sense, the minimum shifting required to ensure non-negative (or positive) eigenvalues for all six Hermitian matrices.
\begin{cor}[Non-negativity]
    \label{cor:non_negativity}
    If $(\ea, \eb, \ec, \ed, \ee, \ef) \in \Tetra(n)$, then
    \begin{equation}
        \label{eq:min_eig_ineqs}
        \ec_n \geq \ea_n + \eb_n, \qquad \ef_n \geq \eb_n + \ed_n, \qquad \ee_n \geq \ed_n + \ec_n, \qquad \ee_n \geq \ea_n + \ef_n.
    \end{equation}
    Let $(x,y,z) \in \mathbb R^{3}$ be such that
    \begin{equation}
        \label{eq:positive_shift}
        x \geq -\ea_n, \qquad
        y \geq -\eb_n, \qquad
        z \geq -\ed_n,
    \end{equation}
    and let $(\tilde \ea, \tilde \eb, \tilde \ec, \tilde \ed, \tilde \ee, \tilde \ef)$ be shifted eigenvalues given by \cref{eq:shifted_tet_eigenvalues}.
    Then
    \begin{equation}
        \label{eq:positive_shift_iff}
        (\ea, \eb, \ec, \ed, \ee, \ef) \in \Tetra(n) \iff (\tilde \ea, \tilde \eb, \tilde \ec, \tilde \ed, \tilde \ee, \tilde \ef) \in \Tetra^{+}(n).
    \end{equation}
\end{cor}
\begin{proof}
    That $(\ea, \eb, \ec, \ed, \ee, \ef) \in \Tetra(n)$ must satisfy \cref{eq:min_eig_ineqs} follows from four applications of the well-known inequality $\ec_n \geq \ea_n + \eb_n$ for the original Horn problem where $C=A+B$ \cite{bhatia2001linear}.
    As a reminder, the smallest eigenvalue $\ex_n$ of a matrix $X$ is characterized by a minimization: $\ex_n = \min_{\norm{v} = 1} \braket{v, Xv}$ which produces
    \begin{align}
        \begin{split}
            &\ec_n = \min_{\norm{v} = 1} \braket{v, C v} = \min_{\norm{v} = 1} \left(\braket{v, A v} + \braket{v, B v}\right) \\
            &\qquad \geq \min_{\norm{v} = 1} \braket{v, A v} + \min_{\norm{w} = 1}\braket{w, B w} = \ea_n + \eb_n.
        \end{split}
    \end{align}
    Four applications of this inequality produces \cref{eq:min_eig_ineqs}.
    Now if $(x,y,z)$ satisfy \cref{eq:positive_shift}, we conclude non-negativity of the shifted eigenvalues, e.g.,
    \begin{equation}
        \forall j : \qquad \tilde \ea_j = \ea_j + x \geq \ea_n + x \geq 0,
    \end{equation}
    Similar statements hold for $\eb,\ec,\ed, \ee$ and $\ef$.
    This establishes the $\Rightarrow$ portion of \cref{eq:positive_shift_iff}.
    The reverse implication holds because $\Tetra^{+}(n) \subseteq \Tetra(n)$ and $\Tetra(n)$ is closed under shifting by \cref{prop:shifting}.
\end{proof}
\begin{rem}
    It is worth emphasizing that our main theorem, \cref{thm:tet_iff}, provides necessary and sufficient conditions for determining membership in $\Tetra^{+}(n)$.
    Nevertheless, by a simple shifting of eigenvalues, \cref{cor:non_negativity} demonstrates how any characterization of $\Tetra^{+}(n)$ also provides a characterization of $\Tetra(n)$.
    However, while $\Tetra(n)$ has numerous symmetries, including the discrete group of order $48$ described by \cref{prop:tet_symmetry}, the subset $\Tetra^{+}(n) \subseteq \Tetra(n)$ has fewer symmetries;
    indeed, the only element in the aforementioned discrete group which preserves non-negativity of eigenvalues is the permutation $(\ea,\eb,\ec,\ed,\ee,\ef) \leftrightarrow (\ed, \eb, \ef, \ea, \ee, \ec)$.
    This difference explains why the inequalities from \cref{thm:tet_iff} do not directly respect the symmetries of $\Tetra(n)$.
\end{rem}
\begin{rem}
    Although this section listed a handful of symmetries of the tetrahedral Horn problem, there remains the possibility of additional symmetries which are unknown at this time.
    For example, in the case of $2\times 2$ Hermitian matrices, an additional class of symmetries on the set of edge-lengths $(a,b,c,e,d,f)$ of tetrahedra, known as the \textit{Regge symmetries}, hold \cite{akopyan2019regge,boalch2007regge,ponzano1968semiclassical}.
    If there exists a tetrahedron with edge lengths $(a,b,c,d,e,f)$, then there also exists a tetrahedron with edge lengths
    \begin{equation}
        (s_1-a,s_1-b,c,s_1-d,s_1-e,f) \quad \text{where} \quad s_1 = \frac{a+b+d+e}{2}.
    \end{equation}
    Combining this symmetry with the permutational symmetries given by \cref{prop:tet_symmetry} (noting eigenvalues of traceless $2 \times 2$ Hermitian matrices are invariant under negation) one concludes the co-existence of a family of six tetrahedra with edge lengths given by \cite[Chapter 9]{varshalovich1988quantum}:
    \begin{align}
        \begin{split}
            &(a,b,c,d,e,f), \\
            &(s_1-a,s_1-b,c,s_1-d,s_1-e,f), \\
            &(a,s_3-b,s_3-c,d,s_3-e,s_3-f), \\
            &(s_2-a,b,s_2-c,s_2-d,e,s_2-f), \\
            &(s_2-d,s_1-e,s_3-f,s_2-a,s_1-b,s_3-c), \\
            &(s_1-d,s_3-e,s_2-f,s_1-a,s_3-b,s_2-c),
        \end{split}
    \end{align}
    where
    \begin{equation}
        s_1 = \frac{a+b+d+e}{2}, \quad
        s_2 = \frac{a+c+d+f}{2}, \quad
        s_3 = \frac{b+c+e+f}{2}.
    \end{equation}
\end{rem}
\begin{rem}
    The Regge symmetry of a tetrahedron can be understood as a consequence of an earlier discovery of a surprising symmetry of Wigner $6j$ symbol for $SU(2)$ originally due to \citeauthor{regge1959symmetry} \cite{regge1959symmetry}.
    The paper due to \citeauthor{boalch2007regge} provides a comprehensive and modern treatment of the Regge symmetry for the $6j$ symbol for $SU(2)$ by viewing it as a consequence of a symmetry of multiplicity-free $6j$-symbols for $SU(3)$ \cite{boalch2007regge}.
    That the Regge symmetry of a tetrahedron emerges from the Regge symmetry of the associated $6j$ symbols can be viewed as a consequence of the semiclassical limit of the latter \cite{ponzano1968semiclassical,roberts1999classical}.
    It has also be shown that the Regge symmetry of Euclidean tetrahedra is actually a scissors-congruence preserving both the volume and Dhen invariants \cite{akopyan2019regge}.
    At the moment, it remains unclear if analogous symmetries for the generalized Wigner $6j$ symbol exist for $n \geq 3$ (although there are conjectured symmetries for the $6j$ symbols for $U_q(\mathfrak{sl}_n)$ which would follow from the eigenvalue hypothesis \cite{morozov2018new,lanina2023tug}).
    It could also be possible for additional symmetries to be present in the tetrahedral Horn problem for $n \geq 3$ which are \textit{not} present in the $6j$ symbol for $U(n)$.
\end{rem}

\section{Results for the original Horn problem}
\label{sec:new_ineq_horn_orig}

The purpose of this section is to describe how our methods can also be used to prove the following result which pertains to the original Horn problem.

\begin{thm}
    \label{thm:tri_iff}
    Let $(\ea,\eb,\ec) \in (\eigsp{n})^3$ be a triple satisfying the trace conditions.
    Then $(\ea,\eb,\ec) \in \Horn^{+}(n)$ if and only if for all $(\ya,\yb) \in (\undomweightsp{n})^2$,
    \begin{equation}
        \label{eq:tri_iff}
        \vp_{\ya}(\ea)\vp_{\yb}(\eb)
        \leq \sum_{\substack{\yc\\c^{\yc}_{\ya\yb} > 0}} \vp_{\yc}(\ec).
    \end{equation}
    Furthermore, if these inequalities are satisfied for all pairs $(\ya,\yb) \in (\undomweightsp{n})^2$ with $\sya + \syb = k$, then
    \begin{equation}
        \label{eq:tri_dist_bound}
        \frac{D(\ea,\eb;\ec)}{\tec} \leq 2 \sqrt{5} n \sqrt{\frac{\ln \polyk{}}{k}}.
    \end{equation}
    where $D(\ea,\eb;\ec)\coloneqq \min_{(A,B) \in \orbit_\ea \times \orbit_\eb} \norm{\eig(A+B)-\ec}_{1}$.
\end{thm}

The proof of this theorem is divided into two parts.
First we prove the necessity of \cref{eq:tri_iff} in \cref{sec:tri_necess}.
Then afterwards we prove the distance bound in \cref{eq:tri_dist_bound} which in turn implies the sufficiency of \cref{eq:tri_iff} upon taking the limit as $k \to \infty$.

\subsection{Necessary conditions}
\label{sec:tri_necess}

The following proposition follows directly from the binomial theorem (\cref{thm:the_binomial_theorem}) and pertains to sums of the quantity $\vp^{\ya\yb}_{\yc}(X,Y)$.

\begin{prop}
    \label{prop:bisummation}
    Let $X,Y$ be $n\times n$ matrices. Then the following identities hold:
    \begin{align}
        \forall \ya,\yb : \quad \sum_{\yc} \vp^{\ya\yb}_{\yc}(X,Y) &= \vp_{\ya}(X)\vp_{\yb}(Y), \\
        \forall \yc : \quad \sum_{\ya,\yb} \vp^{\ya\yb}_{\yc}(X,Y) &= \vp_{\yc}(X+Y).
    \end{align}
\end{prop}
\begin{proof}
    The first identity involving a summation over $\ya$ and $\yb$ is the result of taking the trace of both sides of \cref{thm:the_binomial_theorem}.
    For the second identity, involving a summation over $\yc$, follows from the identity $\sum_{\yc \vdash \sya+\syb} \Pi^{\yc}_{\ya\yb} = \id_{V_{\ya}\ot V_{\yb}}$:
    \begin{align}
        \begin{split}
            \vp_{\ya}(X)\vp_{\yb}(Y)
            &= \Tr[\Phi_{\ya}(X) \ot \Phi_{\yb}(Y)], \\
            &= \sum_{\yc} \Tr[\Pi^{\yc}_{\ya\yb}(\Phi_{\ya}(X) \ot \Phi_{\yb}(Y))], \\
            &= \sum_{\yc} \vp^{\ya\yb}_{\yc}(X,Y),
        \end{split}
    \end{align}
    where the last equality is just the definition of $\vp^{\ya\yb}_{\yc}(X,Y)$.
    Once again, the condition that $\syc = \sya + \syb$ is implicitly present.
\end{proof}

Under the assumption that the matrices $X$ and $Y$ are non-negative, we can conclude that the real number $\vp^{\ya\yb}_{\yc}(X,Y)$ is also non-negative because $\vp^{\ya\yb}_{\yc}(X,Y) = \Tr[\Pi^{\yc}_{\ya\yb} (\Phi_{\ya}(X) \ot \Phi_{\yb}(Y))]$ where $\Pi^{\yc}_{\ya\yb}$ is a projection operator and $\Phi_{\ya}(X),\Phi_{\yb}(Y) \geq 0$ because $X,Y \geq 0$.
Moreover, we note that $\vp^{\ya\yb}_{\yc}(X,Y) = 0$ whenever the Littlewood-Richardson coefficient $c^{\yc}_{\ya\yb}$ vanishes.
These observations imply the following necessary condition for compatibility of eigenvalues for the original Horn problem.
\begin{prop}
    \label{prop:coupling_condition}
    Let $(a,b,c) \in (\eigsp{n})^{3}$.
    Then there exists $(A,B,C) \in (\hermp{n})^{3}$ satisfying $A+B=C$ with eigenvalues $(a,b,c)$ only if for every $(\ya,\yb,\yc) \in (\undomweightsp{n})^{3}$ there exists a $\Gamma^{\yc}_{\ya\yb} \geq 0$ such that $\Gamma^{\yc}_{\ya\yb} = 0$ whenever $c^{\yc}_{\ya\yb} = 0$ and
    \begin{align}
        \forall \ya,\yb : \quad \sum_{\yc} \Gamma^{\yc}_{\ya\yb} &= \vp_{\ya}(\ea)\vp_{\yb}(\eb), \\
        \forall \yc : \quad \sum_{\ya,\yb} \Gamma^{\yc}_{\ya\yb} &= \vp_{\yc}(\ec).
    \end{align}
\end{prop}
\begin{proof}
    The claim follows directly from \cref{prop:bisummation} after identifying $\Gamma^{\yc}_{\ya\yb} = \vp^{\ya\yb}_{\yc}(A,B) \geq 0$.
\end{proof}
\begin{cor}
    \label{cor:row_col_horn_ineqs}
    Let $(a,b,c) \in \Horn{n}$.
    Then,
    \begin{align}
        \forall \ya,\yb : \quad \vp_{\ya}(\ea)\vp_{\yb}(\eb) &\leq \sum_{\substack{\yc\\c^{\yc}_{\ya\yb} > 0}}\vp_{\yc}(\ec),\\
        \forall \yc : \quad \vp_{\yc}(\ec) &\leq \sum_{\substack{\ya,\yb\\c^{\yc}_{\ya\yb} > 0}}\vp_{\ya}(\ea)\vp_{\yb}(\eb).
    \end{align}
\end{cor}
\begin{proof}
    For all triples $(\ya,\yb,\yc)$, if $\Gamma^{\yc}_{\ya\yb} > 0$, then $c^{\yc}_{\ya\yb} > 0$ and thus $\Gamma^{\yc}_{\ya\yb} \leq \min\{\vp_{\ya}(\ea)\vp_{\yb}(\eb), \vp_{\yc}(\ec)\}$.
    The claim follows from \cref{prop:coupling_condition}.
\end{proof}

\begin{rem}
    Using \cref{prop:coupling_condition}, it is possible to prove even stronger inequalities than those appearing in \cref{cor:row_col_horn_ineqs} by examining the set of triples $(\ya, \yb, \yc) \in (\undomweightsp{n})^{3}$ where $c^{\yc}_{\ya,\yb} \neq 0$.
    We illustrate this point with an example.
    Let $k = \sya + \syb = \syc = 4$, then the following is the table of values for $c^{\yc}_{\ya\yb}$.
    \begin{center}
        \newcommand{\zero}{\cellcolor{white}{\color{gray}$0$}}
        \newcommand{\numb}[1]{\cellcolor{black!80!white}{\color{white}$#1$}}
        \begingroup
        \setlength{\tabcolsep}{3pt}
        \begin{tabular}{|cc||c|c|c|c|c||c|c|c||c|c|c|c||c|c|c||c|c|c|c|c|}
            \hline
            $c^{\yc}_{\ya\yb}$ & $\ya$ &
            $\ydiagram{4}$ & $\ydiagram{3,1}$ & $\ydiagram{2,2}$ & $\ydiagram{2,1,1}$ & $\ydiagram{1,1,1,1}$ &
            $\ydiagram{3}$ & $\ydiagram{2,1}$ & $\ydiagram{1,1,1}$ &
            $\ydiagram{2}$ & $\ydiagram{2}$ & $\ydiagram{1,1}$ & $\ydiagram{1,1}$ &
            $\ydiagram{1}$ & $\ydiagram{1}$ & $\ydiagram{1}$ &
            $\ydiagram{}$ & $\ydiagram{}$ & $\ydiagram{}$ & $\ydiagram{}$ & $\ydiagram{}$ \\
            $\yc$ & $\yb$ &
            $\ydiagram{}$ & $\ydiagram{}$ & $\ydiagram{}$ & $\ydiagram{}$ & $\ydiagram{}$ &
            $\ydiagram{1}$ & $\ydiagram{1}$ & $\ydiagram{1}$ &
            $\ydiagram{2}$ & $\ydiagram{1,1}$ & $\ydiagram{2}$ & $\ydiagram{1,1}$ &
            $\ydiagram{3}$ & $\ydiagram{2,1}$ & $\ydiagram{1,1,1}$ &
            $\ydiagram{4}$ & $\ydiagram{3,1}$ & $\ydiagram{2,2}$ & $\ydiagram{2,1,1}$ & $\ydiagram{1,1,1,1}$ \\ \hline\hline
            $\ydiagram{4}$ & &
            \numb{1} & \zero & \zero & \zero & \zero &
            \numb{1} & \zero & \zero &
            \numb{1} & \zero & \zero & \zero &
            \numb{1} & \zero & \zero &
            \numb{1} & \zero & \zero & \zero & \zero
            \\\hline
            $\ydiagram{3,1}$ &  &
            \zero & \numb{1} & \zero & \zero & \zero &
            \numb{1} & \numb{1} & \zero &
            \numb{1} & \numb{1} & \numb{1} & \zero &
            \numb{1} & \numb{1} & \zero &
            \zero & \numb{1} & \zero & \zero & \zero
            \\\hline
            $\ydiagram{2,2}$ &  &
            \zero & \zero & \numb{1} & \zero & \zero &
            \zero & \numb{1} & \zero &
            \numb{1} & \zero & \zero & \numb{1} &
            \zero & \numb{1} & \zero &
            \zero & \zero & \numb{1} & \zero & \zero
            \\\hline
            $\ydiagram{2,1,1}$ &  &
            \zero & \zero & \zero & \numb{1} & \zero &
            \zero & \numb{1} & \numb{1} &
            \zero & \numb{1} & \numb{1} & \numb{1} &
            \zero & \numb{1} & \numb{1} &
            \zero & \zero & \zero & \numb{1} & \zero
            \\\hline
            $\ydiagram{1,1,1,1}$ &  &
            \zero & \zero & \zero & \zero & \numb{1} &
            \zero & \zero & \numb{1} &
            \zero & \zero & \zero & \numb{1} &
            \zero & \zero & \numb{1} &
            \zero & \zero & \zero & \zero & \numb{1}
            \\\hline
        \end{tabular}
        \endgroup
    \end{center}
    Using this table, one sees the inequalities in \cref{cor:row_col_horn_ineqs} are associated to single columns or single rows, e.g., along the $(\ya,\yb) = (\ydiagram{2}, \ydiagram{1,1})$, we obtain
    \begin{equation}
        \vp_{\ydiagram{2}}(X)\vp_{\ydiagram{1,1}}(Y) \leq \vp_{\ydiagram{3,1}}(X+Y) + \vp_{\ydiagram{2,1,1}}(X+Y).
    \end{equation}
    However, using \cref{prop:coupling_condition}, a stronger inequality is obtained by including all columns covered by the rows associated to the right-hand side of the above inequality, e.g.,
    \begin{align}
        \begin{split}
            &\vp_{\ydiagram{3,1}}(X) + \vp_{\ydiagram{2,1,1}}(X) +
            \vp_{\ydiagram{2}}(X)\vp_{\ydiagram{1,1}}(Y) + \vp_{\ydiagram{1,1}}(X)\vp_{\ydiagram{2}}(Y) +
            \vp_{\ydiagram{3,1}}(Y) + \vp_{\ydiagram{2,1,1}}(Y) \\
            &\qquad\leq \vp_{\ydiagram{3,1}}(X+Y) + \vp_{\ydiagram{2,1,1}}(X+Y).
        \end{split}
    \end{align}
    In general, \cref{prop:coupling_condition} describes a linear system of inequality constraints in the variables $\Gamma^{\yc}_{\ya\yb}, \vp_{\ya}(\ea)\vp_{\yb}(\eb),$ and $\vp_{\yc}(\ec)\}$ which, in principle, be solved for $\vp_{\ya}(\ea)\vp_{\yb}(\eb)$ and $\vp_{\yc}(\ec)$ by standard linear quantifier elimination techniques such as Fourier-Motzkin elimination \cite{dantzig1972fourier}.
    While these additional inequalities are certainly interesting to explore, \cref{thm:tri_iff} establishes that for the purposes of resolving the original Horn problem, it is sufficient to restrict our attention to inequalities of the form \cref{cor:row_col_horn_ineqs}.
\end{rem}

\subsection{Distance bound \& sufficiency}
\label{sec:dist_bound_for_orig_horn}

Here we prove the distance bound stated in \cref{thm:tri_iff}.
To begin we fix an \textit{arbitrary} pair $(\ya,\yb)$ with total size $\sya + \syb = k$ and derive multiple bounds on the product $\vp_{\ya}(\ea)\vp_{\yb}(\eb)$.
Afterwards, we will sum over all such pairs and then obtain the claimed bound on $D(\ea,\eb;\ec)$.

First we bound $\vp_{\ya}(\ea)\vp_{\yb}(\eb)$ in two ways and then apply the eigenvalue separation lemma (specifically \cref{eq:weak_eig_sep_varphi}).
The first bound on $\vp_{\ya}(\ea)\vp_{\yb}(\eb)$ is obtained from relaxing \cref{eq:tri_iff} by considering a Young diagram $\yc^{*}$ of size $k$ which maximizes $\vp_{\yc^{*}}(\ec)$ among those $\yc$ such that $c^{\yc}_{\ya\yb}> 0$ for the given $(a,b,c)$.
As there are at most $\polyk{n}$ Young diagrams of size $k$, we conclude
\begin{equation}
    \label{eq:first_ab_bound}
    \vp_{\ya}(\ea)\vp_{\yb}(\eb) \leq \polyk{n} \vp_{\yc^{*}}(\ec).
\end{equation}
The second bound on $\vp_{\ya}(\ea)\vp_{\yb}(\eb)$ is obtained from \cref{prop:random_matrices} which states that for any $\yc$ such that $c^{\yc}_{\ya\yb} > 0$,
\begin{equation}
    \vp_{\ya}(\ea)\vp_{\yb}(\eb)
    = \frac{\dim V_{\ya} \ot V_{\yb}}{c^{\yc}_{\ya\yb} \dim V_{\yc}}
    \tightdoublesubstack{\expect}{X \sim \orbit_{\ex}}{Y \sim \orbit_{\ey}}\,
    \vp^{\ya\yb}_{\yc}(X,Y).
\end{equation}
To relax this into a usable bound we (i) bound the dimensions by $\frac{\dim V_{\ya} \ot V_{\yb}}{c^{\yc}_{\ya\yb} \dim V_{\yc}} \leq \dim (V_{\ya} \ot V_{\yb}) \leq \polyk{2\binom{n}{2}}$, then (ii) bound the average over $X \in \orbit_{\ex}$ and $Y \in \orbit_{\ey}$ with a maximum, then (iii) apply \cref{cor:row_col_horn_ineqs} to get $\vp^{\ya\yb}_{\yc}(X,Y) \leq \vp_{\yc}(X+Y)$ and then finally (iv) specialize to the optimal $\yc \mapsto \yc^{*}$ to produce
\begin{equation}
    \label{eq:second_ab_bound}
    \vp_{\ya}(\ea)\vp_{\yb}(\eb) \leq \polyk{2\binom{n}{2}}
    \tightdoublesubstack{\mathrm{max}}{X \in \orbit_{\ea}}{Y \in \orbit_{\eb}}
    \vp_{\yc^{*}}(X+Y).
\end{equation}
Then we combine the two bounds on $\vp_{\ya}(\ea)\vp_{\yb}(\eb)$ from \cref{eq:first_ab_bound} and \cref{eq:second_ab_bound} and apply the eigenvalue separation bound \cref{eq:weak_eig_sep_varphi} to obtain
\begin{align}
    \label{eq:ab_intermediate_bound}
    \begin{split}
        \vp_{\ya}(\ea)\vp_{\yb}(\eb)
        &\leq \sqrt{\polyk{n^2}
        \tightdoublesubstack{\mathrm{max}}{X \in \orbit_{\ea}}{Y \in \orbit_{\eb}}
        \vp_{\yc^{*}}(X+Y)\vp_{\yc^{*}}(\ec)},\\
        &\leq \polyk{\frac{n^2}{2}}\polyk{\binom{n}{2}}
        \tightdoublesubstack{\mathrm{max}}{X \in \orbit_{\ea}}{Y \in \orbit_{\eb}}
        \frac{\BC(\eig(X+Y), c)^k}{k!}, \\
        &\leq \polyk{n^2-\frac{n}{2}} \frac{\tec^k}{k!}
        \exp\left(- \frac{k}{8 \tec^2} D(a,b;c)^2\right),
    \end{split}
\end{align}
where the last inequality follows from \cref{eq:bhattacharyya_vs_distance} and the definition $D(a,b;c) = \tightdoublesubstack{\mathrm{min}}{X \in \orbit_{\ea}}{Y \in \orbit_{\eb}}\norm{\eig(X+Y) - c)}_1$.

Now we sum both sides of \cref{eq:ab_intermediate_bound} over all pairs $(\ya,\yb)$ with total size $\sya+\syb=k$ and use that there are at most $\polyk{2n}$ such pairs to obtain
\begin{equation}
    \frac{\tec^k}{k!} \leq \polyk{n^2+\frac{3n}{2}} \frac{\tec^k}{k!}
        \exp\left(- \frac{k}{8 \tec^2} D(a,b;c)^2\right),
\end{equation}
where the left-hand side follows from the trace condition $\tec=\tea+\teb$ and
\begin{equation}
    \label{eq:binomial_trace}
    \sum_{\ell=0}^{k} \sum_{\substack{\ya\vdash \ell\\\yb\vdash k-\ell}} \vp_{\ya}(\ea)\vp_{\yb}(\eb) = \sum_{\ell=0}^{k} \frac{\tea^\ell}{\ell!} \frac{\teb^{k-\ell}}{(k-\ell)!} = \frac{(\tea+\teb)^k}{k!},
\end{equation}
whereas the extra factor of $2n$ in the exponent on the right-hand side is obtained because there are at most $\polyk{n}$ Young diagrams with $n$ rows and size at most $k$ and we are summing over a pair $(\ya,\yb)$ with total size $\sya+\syb=k$.
Rearranging then establishes an upper bound on $D(a,b;c)$ of the form
\begin{equation}
    \frac{D(a,b;c)}{\Tr[c]} \leq \sqrt{8\left(n^2 + \frac{3}{2}n\right) \frac{\ln \polyk{}}{k}},
\end{equation}
which implies \cref{eq:tri_dist_bound} because $\sqrt{8\left(n^2 + \frac{3}{2}n\right)} \leq 2 \sqrt{5} n$ for all $n\geq 1$.

\section{The large \texorpdfstring{$n$}{n} limit \& an entropic inequality}
\label{sec:large_n}

The purpose of this section is to show how to leverage the inverse polynomial bound on $\norm{\abcdefsymbol}$ provided by \cref{thm:6j_estimates} to derive inequalities on the entropies of the eigenvalues.
Our strategy will be to prove an upper bound on the $6j$ symbol which depends only on the dimensions of the underlying irreducible representations and the multiplicity spaces between them.
By combining this upper bound with the lower bound provided by \cref{thm:6j_estimates} (specifically the $n$-independent bound in \cref{eq:n_independent_inverse_poly}), we will prove the following.
\begin{thm}
    \label{thm:entropic_ineq}
    Let $(\ea,\eb,\ec,\ed,\ee,\ef) \in \Tetra^{+}(n)$ be tetrahedral eigenvalues.
    Then
    \begin{equation}
        \label{eq:entropic_ineq}
        s(\ec) + s(\ef) - s(\eb) - s(\ee) \geq 0,
    \end{equation}
    where
    \begin{equation}
        \label{eq:little_s_defn}
        s(\ex) \coloneqq - \sum_{i=1}^{n} \ex_i \ln \ex_i.
    \end{equation}
\end{thm}

We begin with our upper bound on the $6j$ symbol.
\begin{lem}
    \label{lem:sixj_ssa_bound}
    Let $(\ya,\yb,\yc,\yd,\ye,\yf) \in (\undomweightsp{n})^{6}$.
    Then
    \begin{equation}
        \norm{\abcdefsymbol}_{2}^{2}
        \leq \left(c_{\ya,\yf}^{\ye}c_{\yb,\yd}^{\yf}c_{\ya,\yb}^{\yc}c_{\yc,\yd}^{\ye}\right)^{\frac{1}{2}} \frac{\dim(V_{\yc}^n) \dim(V_{\yf}^n)}{\dim(V_{\yb}^n) \dim(V_{\ye}^n)}.
    \end{equation}
\end{lem}
\begin{proof}
    Here we provide a diagrammatic proof of the claimed inequality.
    For the sake of clarity, we abbreviate the irreducible representation labels, $V_{\yx}$, and omit the multiplicity space labels on the intermediate steps.
    \begin{equation}
        \includegraphics[width=\textwidth]{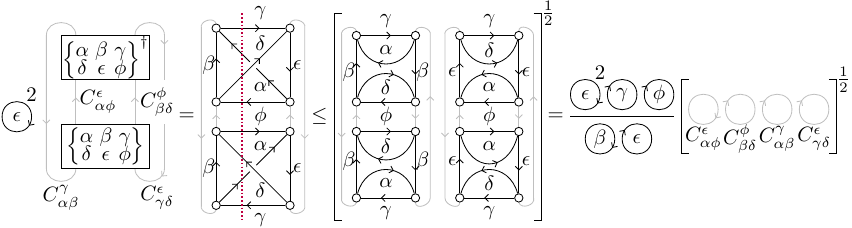}
    \end{equation}
    The left equality follows from the definition of the $6j$ symbol.
    The middle inequality follows from an application of the Cauchy-Schwarz inequality with respect to the annotated contraction.
    The right equality follows from applying Schur's lemma to components of the diagram mapping $V_{\yb}$ to itself on the left and $V_{\ye}$ to itself on the right.
    Note that the factors of $\dim(V_{\ye})^2$ appearing on left-most and right-most sides of the above inequality cancel producing the claimed inequality.
\end{proof}

An immediate application of \cref{lem:sixj_ssa_bound} is to combine it with \cref{thm:6j_estimates}.
If $(\ea,\eb,\ec,\ed,\ee,\ef)$ are tetrahedral eigenvalues, then by \cref{lem:sixj_ssa_bound} and \cref{thm:6j_estimates} there exists a converging sequence of Young diagrams $(\ya_k, \yb_k, \yc_k, \yd_k, \ye_k, \yf_k)$ such that
\begin{equation}
    \label{eq:combined_bound}
    \frac{1}{q(k)} \leq \left(c_{\ya_k,\yf_k}^{\ye_k}c_{\yb_k,\yd_k}^{\yf_k}c_{\ya_k,\yb_k}^{\yc_k}c_{\yc_k,\yd_k}^{\ye_k}\right)^{\frac{1}{2}} \frac{\dim(V_{\yc_k}^n) \dim(V_{\yf_k}^n)}{\dim(V_{\yb_k}^n) \dim(V_{\ye_k}^n)}.
\end{equation}
where $q(k)$ is some polynomial in $k$ independent of the dimension $n$ (see \cref{eq:n_independent_inverse_poly}).
Remarkably, the Littlewood-Richardson coefficients $c_{\ya_k,\yf_k}^{\ye_k}$, $c_{\yb_k,\yd_k}^{\yf_k}$, $c_{\ya_k,\yb_k}^{\yc_k}$, and $c_{\yc_k,\yd_k}^{\ye_k}$ are also \textit{independent} of the dimension parameter $n$ defining the group $U(n)$ (assuming $n$ is at least as large as the maximum depth of any of the involved Young diagrams.
This observation provides us with an opportunity to optimize the above bound by varying the value of $n$.

In particular, it will useful to better understand how the dimension, $\dim(V_{\yx})$, of the irreducible representation $V_{\yx}$ of $U(n)$ grows as $n$ becomes much larger than the number of rows of $\yx$.
In general, for fixed $\yx$, the function sending $n$ to $\dim V_{\yx}$ can grow at an exponential rate with increasing $n$.
One can find estimates for $\dim(V_{\yx})$ in the regime where $n$ is large compared to ${\syx}$ and also the last row $\yx_{\min}$ is large compared to the depth of $\yx$ in paper by \Citeauthor{itzykson1966unitary} \cite[Section III]{itzykson1966unitary}.
Here we use an elegant presentation of this scaling, which relates the asymptotics of $\dim(V_{\yx})$ to the value of $\dim(W_{\yx})$ (see \cite[Lemma 2]{harrow2023approximate} or \cite[Equations (25,26)]{childs2007weak}).
\begin{prop}
    \label{prop:harrow_formula}
    Let $\yx$ be a Young diagram with depth at most $n$. Then
    \begin{equation}
        \frac{\dim(V_{\yx})}{n^{{\syx}}} = \frac{\dim(W_{\yx})}{{\syx}!}\prod_{(i,j) \in \yx} \left(1 + \frac{j-i}{n}\right).
    \end{equation}
\end{prop}
\begin{proof}
    The statement follows from comparing the Weyl dimension formula for $\dim(V_{\yx})$ given in \cref{eq:frob_dim} with the Frobenius dimension formula for $\dim(W_{\yx})$ given in \cite[Corollaries 9.1.5]{goodman2000representations}:
    \begin{equation}
        \label{eq:frob_dim}
        \dim(W_{\yx}) = \frac{{\syx}!\Delta(\yx + \rho)}{\prod_{i=1}^{n}(\yx+\rho)_i!}.
    \end{equation}
    That this formula for $\dim(W_{\yx})$ does not depend on $n$ for $n \geq \ell(\yx)$ and coincides with the hook length formula given earlier is proved by \cite[Corollary 9.1.6]{goodman2000representations}.
    For additional details, see also \cite[Appendix A]{harrow2023approximate}.
\end{proof}
From \cref{prop:harrow_formula} we obtain the following corollary.
\begin{cor}
    \label{cor:harrow_ssa_lim}
    Let $\yc,\yf,\yb,\ye$ be a quadruple of Young diagrams such that
    \begin{equation}
        \label{eq:balanced_sizes}
        {\syc}+{\syf} = {\syb}+{\sye}.
    \end{equation}
    Then
    \begin{equation}
        \lim_{n\to\infty}
        \frac{\dim(V_{\yc}^n) \dim(V_{\yf}^n)}{\dim(V_{\yb}^n) \dim(V_{\ye}^n)} =
        \frac{{\syb}!{\sye}!}{{\syc}!{\syf}!}
        \frac{\dim(W_{\yc}) \dim(W_{\yf})}{\dim(W_{\yb}) \dim(W_{\ye})}
        = \frac{H_{\yb}H_{\ye}}{H_{\yc}H_{\yf}}
    \end{equation}
\end{cor}
\begin{proof}
    From \cref{prop:harrow_formula} one evidently has for any Young diagram $\yx$ the limit
    \begin{equation}
        \lim_{n\to\infty} \frac{\dim(V_{\yx})}{n^{{\syx}}} = \frac{\dim(W_{\yx})}{{\syx}!} = \frac{1}{H_{\yx}}
    \end{equation}
    where $H_{\yx}$ is the product of the hook lengths in $\yx$ as defined by \cref{eq:dim_W_hook_product}.
    The claim then follows because the missing prefactor, $\frac{n^{{\syb}}n^{{\sye}}}{n^{{\syc}}n^{{\syf}}}$, is equal to one by \cref{eq:balanced_sizes}.
\end{proof}
We are now in a position to prove the main result of this section.
\begin{proof}[Proof of \cref{thm:entropic_ineq}]
    As $(\ea,\eb,\ec,\ed,\ee,\ef)$ are tetrahedral, the inverse polynomial bound on $\norm{\abcdefsymbol}$ for tetrahedral $(\ea,\eb,\ec,\ed,\ee,\ef)$ provided by \cref{thm:6j_estimates} implies it must at least be non-zero and thus \cref{eq:balanced_sizes} holds.
    Therefore, by \cref{cor:harrow_ssa_lim} and the discussion surrounding the bound in \cref{eq:combined_bound} we conclude that along the convergent sequence of Young diagrams,
    \begin{equation}
        \label{eq:hook_ratio_limit}
        \lim_{k\to\infty} \frac{1}{k} \ln \left(\frac{H_{\yb_k}H_{\ye_k}}{H_{\yc_k}H_{\yf_k}}\right) \geq 0
    \end{equation}
    provided the limit.
    To prove the limit exists, an application of Stirling's approximation, $\ln n! \approx n \ln n - n$ \cite{robbins1955remark}, to the approximation formula $H_{\yx} \approx \prod_{i=1}^{n}\yx_i!$ produces for any constant $\eta$ the formula
    \begin{equation}
        \frac{1}{\eta}\ln H_{\yx} \approx (\ln \eta - 1) \syx + \sum_{i=1}^{n} \frac{\yx_i}{\eta}\ln \frac{\yx_i}{\eta}
    \end{equation}
    Because ${\syc}+{\syf} = {\syb}+{\sye}$ (\cref{eq:balanced_sizes}), the $(\ln \eta - 1)$ terms will cancel in \cref{eq:hook_ratio_limit}.
    Then using the fact that that $(\yb_k, \yc_k, \ye_k, \yf_k)$ converge in proportion to the eigenvalues $(\eb,\ec,\ee,\ef)$ (with proportionality constant $\eta = \frac{k}{\tee}$), we conclude that \cref{eq:hook_ratio_limit} has a limit and moreover it equals the left-hand side of \cref{eq:entropic_ineq}.
\end{proof}

\begingroup
\setlength{\emergencystretch}{2em}
\printbibliography
\endgroup

\end{document}